\documentclass[12pt]{amsart}

\usepackage{latexsym}
\usepackage{amssymb}
\usepackage{amsmath}
\usepackage{amsfonts}

\usepackage{setspace,  
hyperref, epigraph} 

\usepackage{breakurl}   

\usepackage{enumitem} 

\usepackage{tikz}
\usetikzlibrary{decorations.pathreplacing}

\newtheorem{theorem}{Theorem}[section]
\newtheorem{lemma}[theorem]{Lemma}         
\newtheorem{proposition}[theorem]{Proposition}    
\newtheorem{corollary}[theorem]{Corollary}         

\theoremstyle{definition}
\newtheorem{definition}[theorem]{Definition}   
\newtheorem{example}[theorem]{Example}

\newtheorem{remark}[theorem]{Remark}

\newcommand{\bbR}{\mathbf{R}}

\newcommand{\bbD}{\mathbf{D}}

\newcommand{\bbV}{\mathbf{V}}

\newcommand{\bbG}{\mathbf{G}}

\newcommand{\bbM}{\mathbf{M}}
\newcommand{\bbI}{\mathbf{I}}

\newcommand{\bbL}{\mathbf{L}}

\newcommand{\bbS}{\mathbf{S}}
\newcommand{\bbX}{\mathbf{X}}

\newcommand{\bbQ}{\mathbb{Q}}

\newcommand{\bbB}{\mathbf{B}}
\newcommand{\bbE}{\mathbf{E}}
\newcommand{\bbbY}{\mathbf{Y}}

\newcommand{\bN}{\mathbb{N}}
\newcommand{\bR}{\mathbb{R}}

\newcommand{\cA}{\mathcal{A}}

\newcommand{\cF}{\mathcal{F}}

\newcommand{\cH}{\mathcal{H}}

\newcommand{\cP}{\mathcal{P}}

\newcommand{\cT}{\mathcal{T}}
\newcommand{\cL}{\mathcal{L}}

\newcommand{\ccB}{\mathcal{B}}

\newcommand{\fN}{\mathfrak{N}}

\newcommand{\imp}{\rightarrow}
\newcommand{\eqi}{\longleftrightarrow}
\newcommand{\ra}{\rangle}
\newcommand{\la}{\langle}

\newcommand{\st}{\operatorname{\mathrm{st}}}

\newcommand{\fin}{\operatorname{\mathrm{fin}}}
\newcommand{\dom}{\operatorname{dom}}

\newcommand{\rank}{\operatorname{rank}}

\newcommand{\sh}{\operatorname{\mathrm{sh}}}
\newcommand{\ext}{\operatorname{\mathrm{ext}}}

\newcommand{\as}{{}^{\ast}\!}

\newcommand{\ZF}{\mathrm{ZF}}
\newcommand{\ZFC}{\mathrm{ZFC}}

\newcommand{\SCOT}{\mathrm{SCOT}}

\newcommand{\BST}{\mathrm{BST}}
\newcommand{\IST}{\mathrm{IST}}
\newcommand{\HST}{\mathrm{HST}}

\newcommand{\T}{\mathrm{T}}
\newcommand{\B}{\mathrm{B}}

\newcommand{\AC}{\mathrm{AC}}

\newcommand{\pmba}{\raisebox{-6pt}{\begin{tikzpicture}[scale=.3]
\draw [dashed, decorate, decoration={brace, amplitude=5pt}] (0,0) -- (0,2);
\end{tikzpicture}}}

\newcommand{\pmbb}{\raisebox{-6pt}{\begin{tikzpicture}[scale=.3]
\draw [dashed, decorate, decoration={brace, amplitude=5pt}] (2,2) -- (2,0);
\end{tikzpicture}}}

\newcommand{\pmbaa}{\pmba\hspace{-5pt}\pmba}

\newcommand{\pmbbb}{\pmbb\hspace{-5pt}\pmbb}

\begin{document}
\title [Nonstandard hulls and Loeb measures]
{Constructing nonstandard hulls and Loeb measures in internal set theories}
\author{Karel Hrbacek}
\address{Department of Mathematics\\                City College of CUNY\\                New York, NY 10031\\}                \email{khrbacek@icloud.com}

\author{Mikhail G. Katz}
\address{Department of Mathematics\\    Bar Ilan University\\       Ramat Gan 5290002 Israel\\}                \email{katzmik@math.biu.ac.il}

\date{December 4, 2022}

\begin{abstract}
Currently the two popular ways to practice Robinson's nonstandard
analysis are the \emph{model-theoretic} approach and the
\emph{axiomatic/syntactic} approach.  It is sometimes claimed that the internal
axiomatic approach is unable to handle constructions relying on
external sets.  We show that internal frameworks provide
successful accounts of nonstandard hulls and Loeb measures.  
The basic fact this work relies on is that the ultrapower of the
standard universe by a standard ultrafilter 
is naturally isomorphic to a subuniverse of the internal universe. 
\end{abstract}

\maketitle

\tableofcontents

\section{Introduction}\label{Introduction}

Robinson named his theory ``Non-standard Analysis since it involves and
was, in part, inspired by the so-called Non-standard models of
Arithmetic whose existence was first pointed out by T. Skolem''
(Robinson \cite{R}, p. vii).
Currently there are two popular ways to practice Robinson's nonstandard analysis. 

The
\emph{model-theoretic approach}  encompasses Robinson's enlargements, obtained from the Compactness Theorem (Robinson~\cite{R}), ultrapowers (Luxemburg~\cite{Lux}),  and nonstandard universes based on superstructures (Robinson and Zakon~\cite{RZ}, Chang and Keisler~\cite{CK}).

The \emph{axiomatic/syntactic approach} originated in Hrbacek~\cite{H2} and Nelson~\cite{N}. Nelson's $\IST$ is particularly well known (see for example Robert~\cite{Rob}, Diener and Stroyan~\cite{DS}).
The monograph of Kanovei and Reeken~\cite{KR} is a comprehensive reference for axiomatic nonstandard analysis.

In the model-theoretic approach one works  with various nonstandard universes in $\ZFC$; see Chang and Keisler~\cite{CK}, Section 4.4, for definitions and terminology associated with this framework.
Let $\fN = ( V(X), V(\as X), \ast)$ be  a nonstandard universe.
The collection $\as\, V(X)$ of internal sets in $\fN$ is isomorphic to some bounded ultrapower
\footnote{
More generally, some bounded limit ultrapower.}
 of the superstructure $V(X)$.
It is  expanded to the superstructure $V(\as X)$ that contains also external sets.

Axiomatic systems for the internal part of nonstandard set theory, formulated in
the $\st$-$\in$-language,  are well suited for the development of infinitesimal analysis and much
beyond.  It is generally acknowledged that internal theories are easier to learn and to work with than the model-theoretic approaches.  However, it is sometimes claimed as a shortcoming of the
internal approach that external sets are essential for some of the
most important new contributions of Robinsonian nonstandard analysis
to mathematics, such as the constructions of nonstandard hulls
(Luxemburg~\cite{Lux2}) and Loeb measures (Loeb~\cite{Loeb}); see for example Loeb and Wolff~\cite{LW} p. xiv and \cite{Loeb2} p.\;vii, Loeb~\cite{LV} and Zlato\v{s}~\cite{Z}.
Such claims are not meant to be taken literally.
As $\IST$ includes $\ZFC$ among its axioms, nonstandard universes  and the external constructions in them make sense in $\IST$. 
However,  nonstandard universes have their own notions of standard and internal that are not immediately related to the analogous notions provided by $\IST$. 
Such a duplication of concepts could be confusing and complicates attempts to work directly in superstructure frameworks inside $\IST$.
In any case, the issue in question is whether these constructions can be carried out in internal set theories using the notions that these theories axiomatize.

In this paper we describe an approach to external constructions that is better suited to the conceptual framework provided by internal set theories.
The techniques we employ  have been known for a long time, but their use for the purpose of implementing external methods in internal set theories does not seem to explicitly appear in the literature.
We show that practically all objects whose construction in a superstructure involves external sets
can be obtained in the internal axiomatic setting by these techniques.

We work in \emph{Bounded Set Theory} $\BST$ (see Kanovei and Reeken~\cite{KR}, Chapter~3).
The axioms of $\BST$ are a slight modification of the more familiar axioms of $\IST$. We state them in Section~\ref{old} and follow with a discussion of how \emph{definable} external sets can be handled in $\BST$ as abbreviations.

The not-so-well known fact about $\BST$ is that its universe contains  subuniverses isomorphic to any standard ultrapower (or even limit ultrapower) of the standard universe. 
For every internal set $w$ there is a subuniverse $\bbS_w$ and a simply defined isomorphism of $\bbS_w$ with the ultrapower of the standard universe by a standard ultrafilter generated by $w$.
Thus $\BST$ naturally provides an analog of the internal universe $\as\, V(X)$ of any superstructure.
The subuniverses $\bbS_w$ satisfy some of the axioms of $\BST$, as described in Subsection~\ref{wstandard}; hence one can work with these subuniverses axiomatically, without any reference to their relationship to ultrafilters or ultrapowers.

External subsets of $\bbS_w$ and higher-order external sets built from them 
are not objects of $\BST$,  but
the above-mentioned isomorphism, in combination with the powerful principle of Standardization, provides a natural way to code these ``non-existent'' external sets by standard sets and to imitate the superstructure $V(\as X)$ of external sets. This idea is developed in Subsection~\ref{coding}.

The heart of the paper is Section~\ref{apps}, where we show how to construct nonstandard hulls of standard metric and uniform spaces and Loeb measures.   
We also consider analogous constructions on internal normed spaces,  and outline how one can treat neutrices and external numbers.
From the practical point of view, the best way to use these techniques may be to work out the external constructions  informally, and then code them up by standard sets. 
In principle, any construction involving external sets can be carried out in the framework of $\BST$ by these methods.

Section~\ref{appB} develops the relationship between the subuniverses $\bbS_w$ and ultrapowers, and supplies proofs of the properties of $\bbS_w$ stated in Subsection~\ref{wstandard}.

 In Section~\ref{appA}  we review some ways that have been proposed for doing external constructions in $\IST$  previously and discuss their shortcomings.  The principal difficulty is that they tend to produce objects that are ``too large''  to be suitable for further work (larger than the class of standard elements of any standard set).


\section{The internal set theory $\BST$}\label{old}

We work 
in the framework of $\BST$.\footnote{The \emph{bounded} sets of $\IST$ (those sets that are elements of standard sets) satisfy all of the axioms of $\BST$. Hence all arguments in this paper can be carried out in $\IST$, albeit less directly.}
The theory $\BST$ is formulated in the $\st$-$\in$-language. It is a conservative extension of $\ZFC$. 

Quantifiers with the superscript $\st$ range over standard sets.
Quantifiers with the superscript $\st\! \fin$ range over standard finite sets.
The axioms of  $\BST$ are, in addition to $\ZFC$ (the $\ZFC$ axiom schemata of Separation and Replacement apply to $\in$-formulas only): 
\begin{itemize}
\item
$\B$  (Boundedness)  $\quad \forall x \, \exists^{\st} y \,(x \in y)$. 
\item
$\T$ (Transfer) 
Let $\phi (v)$ be an $\in$-formula with standard parameters. Then
$$\forall^{\st} x\; \phi(x)  \imp  \forall x\; \phi(x).$$
 \item
$\mathrm{S}$ (Standardization)  Let $\phi (v)$ be an $\st$-$\in$-formula with arbitrary parameters. Then
$$\forall A\, \exists^{\st}  S\; \forall^{\st} x\; (x \in S \eqi x \in A \,\wedge\,\phi(x)). $$
\item
$\mathrm{BI}$  (Bounded Idealization)
Let $\phi(u,v)$ be an $\in$-formula with arbitrary parameters. For every standard set $A$
\[
\forall^{\st\!\fin} a \subseteq A \;\exists y\; \forall x \in a\;
\phi(x,y) \eqi \exists y \; \forall^{\st} x \in A \; \phi(x,y).
\]
\end{itemize}
See the references Kanovei and Reeken~\cite{KR} and Fletcher et al.~\cite{F} for motivation and more detail.

An equivalent existential version of Transfer  is $$\exists x\; \phi(x)  \imp  \exists^{\st} x\; \phi(x),$$  for $\in$-formulas $\phi $ with standard parameters.
Yet another equivalent version, easily obtained by induction on the logical complexity of the $\in$-formula $\phi(v_1,\ldots,v_k)$ (with standard parameters), is
$$\forall^{\st} x_1,\ldots, x_k\; [\,\phi (x_1,\ldots, x_k) \eqi \phi^{\st} (x_1,\ldots, x_k)]$$
where $\phi^{\st}$ is the formula obtained from $\phi$ by relativizing all quantifiers to $\st$
(that is, by replacing each occurrence of $\exists  $ by $\exists^{\st}   $
and  each occurrence of $\forall   $ by $\forall^{\st } $).

As usual in mathematics, symbols $\bN$ and $\bR$ denote respectively the set of all natural numbers and the set of all reals. In internal set theories there are two ways of thinking about them.
In the ``internal picture''  $\bR$ is viewed as the usual set of reals in which the predicate $\st$ singles out some elements as ``standard''; similarly for any infinite standard set. This is the view familiar from Nelson~\cite{N}.
In the ``standard picture'' the usual set $\bR$ is viewed as containing, in addition to its standard elements, also fictitious, ideal elements. See Fletcher et al.~\cite{F} for  further discussion.

Mathematics in $\BST$ can be developed in the same way as in $\IST$.
In particular, real numbers $r, s$ are \emph{infinitely close} (notation: $r \simeq s$) if $|r - s| < 1/n$ holds for all standard $n \in \bN\setminus \{0\}$. A real number $r $ is an \emph{infinitesimal} if $r \simeq 0$, $r \ne 0$. It is \emph{limited} if $|r| < n$ for some standard $n \in \bN$.
We recall that for every limited $x \in \bR$ there is a unique standard $r \in \bR$ such that $r \simeq x$;
it is called the \emph{shadow} (or \emph{standard part}) of $x$ and denoted $\sh(x)$; 
we also define $\sh(x) = +\infty$ when $x$ is unlimited, $x > 0$, and 
$\sh(x) = -\infty$ when $x$ is unlimited, $x < 0$.

All objects whose existence is postulated by $\BST$ are sets, sometimes called \emph{internal sets} for emphasis. 
The Separation axiom holds for $\in$-formulas only.
But it is common practice in the literature based on the internal axiomatic approach to introduce \emph{definable external sets} as convenient abbreviations (see eg. Vakil~\cite{V}, Diener and Stroyan~\cite{DS}).
One can enrich the language  of the theory by names for extensions of arbitrary $\st$-$\in$-formulas and in this way talk about $\st$-$\in$-definable subclasses of the universe of all (internal) sets. 
We note that
 (a) this does not amount to a formalization of a new type of entity called ``external set,'' which is a  more complicated task; 
and (b) this does not amount to informal use of the term ``external set,'' either (in the sense of relying on a background formalization).
It is similar to the way set theorists routinely employ classes in $\ZFC$ (the class of all sets, the class of all ordinals). Such classes serve as  convenient shortcuts in mathematical discourse because one can work with them ``as if'' they were objects, but they can in principle be replaced by their defining formulas.
Example~\ref{classuse} below illustrates this familiar procedure.

 Let $\phi (v)$ be an $\st$-$\in$-formula with arbitrary parameters. 
We employ dashed curly braces to denote the \emph{class} $\pmbaa x  \,\mid\, \phi (x) \pmbbb.$  
We emphasize that this is  merely a matter of convenience; the expression $z \in \pmbaa x 
\,\mid\, \phi (x) \pmbbb$ is just another notation for $\phi (z)$.
Usually we denote classes by  boldface characters.
Those classes that are included in some set are \emph{external sets}.
If there is a set $A$ such that $\forall x\,(x \in A \eqi \phi(x))$, then the class $\pmbaa x  \,\mid\, \phi (x) \pmbbb$ can be identified with the set $A$.\footnote{
In the model-theoretic approach \emph{external sets} are by definition the sets that are not internal. In the axiomatic approach it is customary to view internal sets as a special case of external sets.
}
Monads and galaxies are some familiar examples of external sets that are usually not sets. 
Let $(M, d) $ be a  metric space: the \emph{monad} of $a \in M$ is
$\bbM(a)=\pmbaa x \in M \,\mid\, d(x,a) \simeq 0\pmbbb$,
and the \emph{galaxy} of $a \in M$ is
$\bbG(a) = \pmbaa x \in M \,\mid\, d(x,a) \text{ is limited}\pmbbb$.
Some useful \emph{proper classes} (ie, classes that are not external sets) are $\bbV = \pmbaa x \,\mid\,x = x \pmbbb$ (the universe of all (internal) sets),\footnote{The symbol $\bbI$ is often used for this purpose in the literature, but in the context of $\st$-$\in$-theories, where all sets are internal, the notation $\bbV$ seems more appropriate.} and
$\bbS =  \pmbaa x \,\mid\, \st(x) \pmbbb$ (the universe of all standard sets).

\begin{example}\label{classuse}
Let $M$, $f: M \to M$ and $a \in M$ be standard.
A convenient way of defining continuity is as follows: 
\emph{The function $f$ is continuous at $a$ if $f[\bbM(a) ]\subseteq \bbM(f(a)).$}
But the use of external sets in this definition can be eliminated by rephrasing it as: 
\emph{The function $f$ is continuous at $a$ if for all $x$, $ d(x,a) \simeq 0$ implies  $d(f(x),f(a)) \simeq 0$.}
\end{example}

Definable external collections of internal sets  are adequately handled in $\BST$ in this manner; we refer to Vakil~\cite{V} for a thorough discussion. 
Difficulties arise  only when higher level constructs on external sets are needed, such as quotient spaces and power sets. We show how to handle such difficulties in Subsections~\ref{coding},~\ref{nshulls}, ~\ref{inthulls}, and~\ref{loeb}.


\section{Subuniverses of the universe of $\BST$}\label{subuniverses}

\subsection{$w$-standard sets}\label{wstandard}
Let us fix a set $w$ and a standard set $I$ such that $w \in I$.\footnote{
The Boundedness axiom guarantees that some such $I$ exists. This is one of the reasons we prefer to work with $\BST$ rather than $\IST$.
} 

\begin{definition} \label{def1}
A set $x$ is called $w$-\emph{standard} (notation: $\st_w(x)$) if $x =
f(w)$ for some standard function $f$ with domain $I$.  Next, we let
$\bbS_w = \pmbaa x \,\mid\, \st_w(x) \pmbbb$ be the universe of all
$w$-standard sets.  The notion of $w$-standardness depends only on
$w$, not on the choice of the standard set $I$.
\end{definition}

\begin{proposition}\label{propos}
(a) \quad $\quad \forall x\, (\st(x) \imp \st_w(x))$. 

(b) \quad $\forall x\, (\st_w(x) \imp \st(x))$ holds if and only if $w$ is standard.

(c) \quad$\st_w(f) \,\wedge \, \st_w(x) \,\wedge\, x \in \dom f\imp \st_w(f(x))$.

\end{proposition}

\begin{proof}
(a) Let $x$ be standard; we have $x = c_x(w)$ where $c_x$ is the constant function with value $x$.

(b) If $w$ is standard, then every $f(w)$ for standard $f$ is standard. If $w$ is nonstandard, let $f(i) =i$ be the identity function on $I$. Then $f(w) = w$ is $w$-standard but not standard.

(c) Assume $\st_w (f)$, $\st_w(x)$ and $x \in \dom f$. Then there are standard functions $F$, $G$ on $I$ such that 
$f = F(w)$ and $x = G(w)$. Define a  function $H$ on $I$ by
$$
H(i) = F(i) (G(i)) \text{ when the right side is defined; } H(i) = \emptyset \text{ otherwise.}
$$
Then $H$ is a standard function on $I$ and $H(w) = f(x)$.
\end{proof}

\begin{definition} \label{good}
A set $w$ is \emph{good} if 
there is $\nu \in \bN$ such that $\nu$ is $w$-standard but not standard.
\end{definition}

In particular,  if $\nu \in \bN$ is nonstandard, then $w = \nu$ is good and, more generally, $w = \la \nu, z \ra$ is good for any set $z$.

The following facts are immediate consequences of known results  (see Kanovei and Reeken~\cite{KR}, Sections~3.3, 6.1, 6.2, esp. Theorem 6.2.6).  For easy reference we give the proofs in Section~\ref{appB}.
Quantifiers with the superscript $\st_w$ range over $w$-standard sets.

\begin{proposition} \label{univ}
\begin{enumerate}
\item
\emph{(Transfer from $w$-standard sets)}
Let $\phi$ be an $\in$-formula with $w$-\emph{standard parameters}. Then
$$\forall^{\st_w} x\; \phi(x)  \imp  \forall x\; \phi(x).$$

\item
\emph{(Countable Idealization into $w$-standard sets)}

Let $\phi$ be an $\in$-formula with $w$-standard parameters. If $w$ is
good, then
\[
\quad \forall^{\st} n \in \bN\;\exists x\; \forall m \in \bN\; (m \le
n \;\imp \phi(m,x))\eqi \exists^{\st_w} x \; \forall^{\st} n \in \bN
\; \phi(n,x).
\]
\emph{In other words, $(\bbS_w, \bbS, \in)$ satisfies Countable Idealization.}

\item
\emph{(Representability)}\\
Let $\phi(v)$ be an $\in$-formula with standard parameters.
If $I $ is standard, $w \in I$, $x$ is $w$-standard, and $\phi(x)$ holds, then there is a standard function $f$ with $\dom f =I$ such that $x = f(w)$ and  $\phi (f(i))$ holds for all $i \in I$.
\end{enumerate}
\end{proposition}

An  equivalent formulation of (1) is \emph{Transfer into $w$-standard sets}:  
$$\exists x\; \phi(x)  \imp  \exists^{\st_w} x\; \phi(x).$$

Another equivalent formulation for $\phi(v_1,\ldots,v_r)$ with $w$-standard parameters is $$\forall^{\st_w} x_1,\ldots, x_r\; [\,\phi (x_1,\ldots, x_r)\eqi 
\phi^{\st_w} (x_1,\ldots, x_r)],$$
where $\phi^{\st_w}$ is the formula obtained from $\phi$ by relativizing all quantifiers to $\st_w$.
In yet other words, $(\bbV, \bbS_w,  \in)$ satisfies Transfer. It also satisfies Boundedness and, for good $w$,   Bounded Idealization
(see Kanovei and Reeken~\cite{KR},  Theorem 6.1.16 (i)), but not Standardization (ibid, Theorem 6.1.15).
Although  we do not need these results in this paper, together they show that $(\bbV, \bbS_w,  \in)$  satisfies all the axioms of $\BST$ except Standardization.
Note that here $\bbS_w$ plays the role of a new ``thick standard universe.'' 
On the other hand, $(\bbS_w , \bbS , \in)$ satisfies Transfer, Boundedness, Standardization and Countable Idealization; here $\bbS_w$ plays the role of a new 
 ``thin internal universe.'' 

Idealization can be strengthened from $\bN$ to sets of cardinality $\kappa$ if  $w$ is chosen carefully. We  leave the technical definition of $\kappa^+$-\emph{good} sets to  Subsection~\ref{universproofs}
(see Definition~\ref{kappasat}). For our applications we need only  to know that for every standard uncountable cardinal $\kappa$ and every $z$ there exist $\kappa^+$-good $w$ so that $z$ is $w$-standard, a result which is also proved there.

\begin{proposition}\label{satuniv}
\emph{(Idealization into $w$-standard sets over sets of cardinality $\le \kappa$)}
Let $\phi$ be an $\in$-formula with $w$-standard parameters. If $w$ is $\kappa^+$-good, then
for every standard set $A$ of cardinality 
$\le \kappa$
$$\forall^{\st\! \fin} a \subseteq A \;\exists y\; \forall x \in a\;
\phi(x,y) \eqi \exists^{\st_w} y \; \forall^{\st} x \in A \; \phi(x,y).$$
\end{proposition}

Propositions~\ref{univ} and \ref{satuniv} do not exhaust the
properties of the universes $\bbS_w$; for a list of further useful
principles see Kanovei and Reeken \cite{KR}, Theorem 3.3.7.


\subsection{Coding external sets} \label{coding}

Definition~\ref{def1} provides a natural way to represent $w$-standard sets by standard functions:
A $w$-standard $\xi \in \bbS_w$ is \emph{represented} by any standard $f \in \bbV^I$ such that $f(w) = \xi$.
Note that every $\xi$ has a proper class of representations. This causes some technical difficulties (see Section~\ref{appB}), which for our purposes are best resolved by fixing a \emph{universal standard set} $V$ so that all objects of interest are subsets of $V$ or relations on $V$; usually one requires $\bR\subseteq V$.  
 By Representability, for every $\xi \in V \cap \bbS_w$ there exists a  standard $f \in V^I$ such that $f(w) = \xi$. Moreover,
if $\psi$ is an $\in$-formula with standard parameters and $\psi(\xi)$ holds, then  $f$ can be chosen so that  $\psi (f(i))$ holds for all $i \in I$.
The representation  makes possible a coding of the external subsets of $V \cap \bbS_w$ by standard sets, and the coding process can be continued to the putative higher levels of the external cumulative hierarchy.
This process is enabled by the principle of Standardization.

\begin{definition}
Let $\phi(v)$ be a formula in the $\st$-$\in$-language, with arbitrary parameters. We let
$${}^{\st} \{ x \in A \,\mid\, \phi (x) \}$$
 denote the  standard set 
 $S$ such that  $\forall^{\st} x\,(x \in S \eqi x \in A \,\wedge\,\phi(x))$.
\end{definition}
The principle of Standardization postulates the existence of this set and Transfer guarantees its uniqueness.
Also by Transfer, if $\psi$ is any $\in$-formula with standard parameters and
$\forall^{\st} x \in A\; (\phi(x) \imp \psi(x))$, then  $\forall x \in S\; \psi(x)$.

In particular, if $\phi(u.v)$ is a formula in the $\st$-$\in$-language with arbitrary parameters, $A$, $B$  are standard, and  for every standard $x \in A$ there is a unique standard $y \in B$ such that $\phi(x, y)$, then there is a  unique standard function $F: A \to B$ such that $\forall^{\st} x \in A\;\phi (x, F(x))$ holds.
It suffices to let $F = {}^{\st}\{ \la x, y \ra \in A \times B \,\mid\, \phi(x,y) \}$.

\begin{definition}
The $w,V$-\emph{code} of an external set $\bbX\subseteq  V \cap \bbS_w$ is the standard set  
$$\Psi_{w,V} (\bbX) =X =     {}^{\st}\{ f \in V^I\,\mid\,    f(w) \in \bbX \}.$$
\end{definition}
In words, $\Psi_{w,V} (\bbX) $ is the standard set whose standard elements are precisely the standard $f\in V^I $ with $f(w) \in \bbX$.

We omit the subscripts $w$ and/or $V$ when they are understood from the context.

The coding is trivially seen to preserve elementary set-theoretic operations.  

\begin{proposition}\label{rulesforpsi}
For any  external sets $\bbX_1, \bbX_2 \subseteq V \cap \bbS_w$:
\begin{enumerate}
\item
$\Psi(\emptyset) = \emptyset$, \quad
$\bbX_1 \subseteq \bbX_2 \eqi \Psi(\bbX_1) \subseteq \Psi(\bbX_2) $;
\item
$\Psi(\bbX_1 \cup \bbX_2) =\Psi(\bbX_1) \cup \Psi(\bbX_2) $,\quad
$\Psi(\bbX_1 \cap \bbX_2) = \Psi(\bbX_1) \cap \Psi(\bbX_2) $;
\item
$\Psi(\bbX_1 \setminus \bbX_2) = \Psi(\bbX_1) \setminus \Psi(\bbX_2) .$ 
\end{enumerate}
\end{proposition}

The coding preserves infinite unions and intersections as well.  
An externally countable sequence of external subsets of $V$ can be viewed in $\BST$ as an external subset $\bbX$ of $\bN \times V$, with $\bbX_n = \pmbaa x \,\mid\, \la n, x \ra \in \bbX \pmbbb$ for standard $n \in \bN$.
Let $\la X_n \,\mid\, n \in \bN \ra$ be the standard sequence such that $X_n = \Psi(\bbX_n)$ holds for all standard $n$ (its existence follows from Standardization). 
Then $$\Psi(\bigcup_{n \in \bN \cap \bbS} \bbX_n) = \bigcup_{n \in \bN} X_n$$
because if $f \in  \bigcup_{n \in \bN} X_n$ is standard, then   the least $n \in \bN$ such that $f \in X_n$ is standard. 
Similarly, $$\Psi(\bigcap_{n \in \bN \cap \bbS} \bbX_n) = \bigcap_{n \in \bN} X_n$$
because if a standard $f \in \Psi(\bigcap_{n \in \bN \cap \bbS} \bbX_n)$, then    $f \in X_n$ for all standard $n \in \bN$, and hence $f \in \bigcap_{n \in \bN} X_n$ by Transfer. 

It is convenient to relax the definition of coding so that every standard $S \subseteq V^I$ is a code of some external $\bbX  \subseteq V \cap \bbS_w$.

\begin{definition}
A standard set  $S$ \emph{codes} $\bbX \subseteq V \cap \bbS_w$ if 
 $f(w) \in \bbX$ for each standard $f \in S$ and
for each $\xi \in \bbX$ there is some standard $f \in S$ such that $f(w) = \xi$.
\end{definition}
If $S$ {codes $\bbX$, then
\[
\Psi(\bbX) = {}^{\st} \{f \in V^I \,\mid\, f(w) = g(w) \text{ for some
  standard } g \in S\}.
\]
 A code of $\bbX$ has to contain a representation for each $\xi \in \bbX$, but not necessarily all such representations. 

We note that  coding is independent of the choice of the universal standard set in the following sense:
If $\bbX  \subseteq V_1 \cap \bbS_w$ and $V_1 \subseteq V_2$, then $S$ codes $\bbX$ viewed as a subset of   $V_1 \cap \bbS_w$ iff $S$ codes $\bbX$ viewed as a subset of   $V_2 \cap \bbS_w$.
Hence the exact choice of $V$ is usually of little importance.

For any set $x$,  let $c_x$ be the  constant function  with value $x$. 
The informal  identification of $x$ with $c_x$  enables the identification of a standard set $A$ with  a code  \;$ {}^{\st}\{ c_x  \,\mid\,  x \in A \} = \{ c_x  \,\mid\,  x \in A \}$ for $A \cap \,\bbS$.
On the other hand, the standard set   $A^I$ is a code for $A \cap \,\bbS_w$.

Every standard  $S \subseteq V^I$ is a code of the unique external set 
$$\bbX_S = \pmbaa \xi \in V \cap \bbS_w \,\mid\, \xi = f(w)  \text{ for some standard } f_w \in S\pmbbb.$$
Such $S$ codes a subset of an external set $\bbX $
iff $\forall^{\st} f\, (f \in S \imp f(w) \in \bbX)$ iff $S \subseteq \Psi (\bbX) $.
Consequently, it would make sense to interpret the power set of $\widetilde{X}=\Psi (\bbX) $ as a code for the ``non-existent'' external power set $\cP^{\ext} (\bbX)$ of $\bbX$, and the standard subsets of $\cP(\widetilde{X})$ as codes for the ``external subsets of $\cP^{\ext} (\bbX)$.''  Since $\BST$ does not allow collections of external sets, this last remark cannot be made rigorous in it, but one can proceed ``as if'' such higher order external sets existed.

Intuitively, there is a hierarchy of external sets built up over $V \cap \bbS_w$:

$\boldsymbol{\cH}_1 = \cP^{\ext}(V \cap \bbS_w)$ and  $\boldsymbol{\cH}_{n+1} = \cP^{\ext} (\boldsymbol{\cH}_n)$ for standard $n \in \mathbb N$ \\
(we stop here to avoid further complications).
This hierarchy  cannot be formalized in $\BST$. But the corresponding hierarchy over $V^I$ is well-defined:

$H_1 = \cP(V^I)$ and  $H_{n+1} = \cP (H_n)$ for standard $n \in \mathbb N$.\\
In $\BST$ one can work legitimately in the latter hierarchy while keeping in mind 
that the coding establishes a (many-one) correspondence between $\boldsymbol{\cH}_n$ and 
$H_n \cap \bbS$.
Both hierarchies and their relationship could be formalized in $\HST$, a conservative extension of $\BST$ to a theory that encompassess also external sets  (Hrbacek~\cite{H2}, Kanovei and Reeken~\cite{KR}).

Note that the coding process treats elements $\xi$ of $V \cap \bbS_w$ as \emph{individuals}; they are not coded by $\Psi$. 
Thus $\xi$ is \emph{represented} by any $f$ with $f(w) = \xi$, but $\Psi(\xi)$ is undefined. The set
$\{\xi\} \subseteq V\cap \bbS_w$ is coded by $\{f\}$, even when  $ \{\xi\} \in V\cap \bbS_w$.
In particular,
if $\bbR$ is a binary relation on $V \cap \bbS_w$, then the  code for $\bbR$ is the standard binary relation 
$$ \Psi (\bbR) =R =    \,{}^{\st}\{ \la f , g  \ra  \,\mid\,  f, g \in V^I \,\wedge\, \la f(w), g(w)\ra \in \bbX \}.$$ 
Similarly for  functions and relations of higher arity.

Another variant of coding represents  each $\xi \in \bbX$ by a single object. 

\begin{definition} \label{coding2}
 Let $f \in V^I$ be standard. We let
$$ f_{w,V} = f_w  ={}^{\st}\{g \in V^I\,\mid\, g(w) = f(w) \}.$$

We define standard sets 
$V^I/w  = {}^{\st}\{ f_w \,\mid\, f \in V^I\}$  and, for $\bbX  \subseteq V \cap \bbS_w,$
\begin{align*}
\widetilde{\Psi}_w (\bbX) &= {}^{\st} \{ f_w \in V^I/w \,\mid \, f(w) \in \bbX\} \\
&= {}^{\st} \{ F \in V^I/w \,\mid\, \exists^{\st} f \in V^I\, (F = f_w \,\wedge\,  f(w) \in \bbX) \} .
\end{align*}

The coding $\widetilde{\Psi}_w$ is one-one. 
For standard $X \subseteq V^I/w $ we let 
$$\widetilde{\boldsymbol{\Psi}}_{w}^{-1} (X) = \bbX =
\pmbaa \xi \in V \cap \bbS_w\,\mid\, \exists^{\st} f \in V^I\,
(  f_w \in X \,\wedge\, f(w) = \xi  )\pmbbb .$$
\end{definition}

One can obtain $\widetilde{\Psi}_w (\bbX) $ from  $\Psi_w (\bbX) $ and vice versa:\\
$\widetilde{\Psi}_w (\bbX) = {}^{\st} \{ f_w \,\mid\, f \in \Psi_w(\bbX)\}$ and  
$\Psi_w(\bbX) =  {}^{\st} \{  f \in V^I \,\mid\, f_w \in \widetilde{\Psi}_w (\bbX) \}.$\\
Proposition~\ref{rulesforpsi} and the subsequent paragraph hold with $\Psi$ replaced by~$\widetilde{\Psi}$.

We define the hierarchy $\widetilde{H}_1 = \cP(V^I/w)$ and  $\widetilde{H}_{n+1} = \cP (\widetilde{H}_n)$ for standard $n \in \mathbb N$.  
The coding $\widetilde{\Psi}_w$ maps $\boldsymbol{\cH}_1$ onto $\widetilde{H}_1 \cap \bbS$. It extends informally to higher levels by 
$\widetilde{\Psi}_w(\bbX) = {}^{\st} \{ \widetilde{\Psi}_w(\bbbY) \,\mid\, \bbbY \in \bbX\}$ and provides a one-one correspondence between 
$\boldsymbol{\cH}_n$ and  $\widetilde{H}_n \cap \bbS$.


\section{Nonstandard hulls and Loeb measures in $\BST$}\label{apps}

\subsection {Nonstandard hulls of standard metric spaces}\label{nshulls}

Let $\bR$ be the field of real numbers and let $(M, d)$  be a standard metric space, so that $M$ is a standard set and the distance function is a standard mapping $d: M \times M \to \bR$.
A point $x \in M$ is \emph{finite} if $d(x,p)$ is limited for some (equivalently, for all) standard $p \in M$.

Fix an unlimited integer $w \in \bN$. 
We define the standard set $B_w$ by
$$B_w =\, {}^{\st} \{ f \in M^{\bN} \,\mid\, f(w) \text{ is a finite point of }M\}.$$

The standard relation $E_w$ on $B_w$ is defined by
$$E_w = \, {}^{\st} \{ \la f, g \ra \in B _w\times B_w \,\mid\, d(f(w) , g(w))\simeq 0\}.$$
Clearly $E_w $ is reflexive, symmetric and transitive on standard elements of $B_w$; it follows by Transfer that 
$E_w$ is an equivalence relation on $B_w$.
We denote the equivalence class of $f$ modulo $E_w$ by $f_{E_w}$.

By Standardization, there is a standard  function $D_w: B_w \times B_w \to \bR$ determined by  the requirement that
$D_w ( f, g ) = \sh ( d (f(w) , g(w) ))$
 for all standard $f, g\in B_w$.

For standard $f, g\in B_w$ the distance 
$$d (f(w) , g(w) )) \le 
d(f(w), f(0)) + d(f(0), g(0)) + d(g(0), g(w))$$
is limited, so $\sh ( d (f(w) , g(w) )) \in \bR$.

For standard $f,g,h \in B_w$ clearly 
$D_w ( f, g ) = D_w ( g,f ) $, 
$$D_w ( f, h ) \le D_w ( f, g ) + D_w ( g, h) , \text{ and }
D_w ( f, g )  = 0 \text{  iff } \la f, g \ra \in E_w.$$
Also, for standard $f,g,f',g' \in B_w$ we have 
\begin{equation*}\bigl( \la f, f'\ra \in E_w \,\wedge \, \la g, g' \ra \in E_w \bigr)\imp 
D_w(f,g) = D _w(f', g').
\end{equation*}
 By Transfer these properties hold for all $f,g ,h, f',g'  \in B_w$.

We observe that, for the natural choice $V = M \cup \bR$, $B_w$ is a code for the external set 
$$\bbB_w  = \pmbaa x \in M \cap \bbS_w  \,\mid\,  x \text{ is finite} \pmbbb,$$
$E_w$ is nothing but a code for the external equivalence relation 
$$\bbE_w  = \pmbaa \la x, y \ra \in \bbB_w \times \bbB_w   \,\mid\, d(x, y) \simeq 0 \pmbbb,$$
and $D_w$  is a code for the external function $\bbD_w: \bbB_w \times \bbB_w \to \bR \cap \bbS$ defined by 
$\bbD_w = \pmbaa \la x, y, r \ra \,\mid\, r = \sh(d(x,y))\pmbbb$. 

The  construction of the nonstandard hull of $(M,d)$ by the usual external method would  form the quotient space $\bbB _w/\bbE_w$ of $\bbB_w$ modulo $\bbE_w$.
This step cannot be carried out in $\BST$. For $x \in \bbB_w$, the equivalence class of $x$ modulo $\bbE_w $ is $\bbM_w(x) = \pmbaa z \in M \cap \bbS_w \,\mid\, d(x,z) \simeq 0 \pmbbb$, but the 
collection  of the classes $\bbM_w(x)$ for all $x \in \bbB_w$ is not supported by $\BST$.
However, for standard $f \in B_w$, the set $f_{E_w}$ is a code of $\bbM_w(x)$ when $f(w) = x$.
Therefore
$\bbB_w /\bbE_w$ can be replaced by  $B_w/E_w$, which is a standard set in $\BST$.
The standard elements of  $B_w/E_w$ are precisely the codes of the monads $\bbM_w(x)$ for $x \in \bbB_w$. Hence $B_w/E_w$ is a code (as discussed in Subsection~\ref{coding}) of the ``non-existent'' 
(in $\BST$) quotient space $\bbB _w/\bbE_w$.
 
We stress that while external sets serve as a motivation for our constructions, they are not actually used in them; the constructions deal only with sets of $\BST$. The same applies to the rest of this section.  

We let  $\widehat{M}_w = B_w / E_w $ be the standard quotient space of $B_w$ modulo $E_w$.
From now on we often omit the subscript $w$ when it is understood from the context.

The  function $D $ factors by $E$ to a (standard) function $\widehat{D} =  D / E$  on $\widehat{M}$, defined by  
$\widehat{D} ( f_E, g_E ) = D(f, g)$
(as shown above, the value of~$\widehat{D}$ is independent of the choice of representatives $f$, $g$).
It is clear that $\widehat{D}$ is a metric on $\widehat{M} $.

The embedding $c$ of $M$ into $\widehat{M}$ is via constant functions: 
for $x \in M$, $c(x) = (c_x)_E$ where $c_x$ is the constant function on $\bN$ with value $x$.
Trivially, the embedding $c$ preserves the metrics.
We  identify $M$ with its image in $\widehat{M}$ under this embedding.

We emphasize that the structure $(\widehat{M} , \widehat{D} \,)$ depends on the choice of the parameter $w$, an unlimited integer. 
A metric space has a unique completion up to isometry, but it may have many non-isometric nonstandard hulls.

\begin{example}\label{exx1}
Let $M = \bbQ$ be the set of rationals and $d(x,y) = |x-y|$ be the usual metric on $\bbQ$.
We fix an unlimited $w\in \bN$. We note that, for standard $f$,
$$f \in B_w \eqi  f \in \bbQ^{\bN} \,\wedge\, f(w)  \text{ is limited},$$
and for standard $f,g \in B_w$
\[
\begin{aligned}
\la f, g \ra \in E_w & \eqi f(w) \simeq g(w) 
\\&\eqi f(w), g(w) \in \bbM(a)
\text{ for } a = \sh(f(w)) = \sh(g(w)).
\end{aligned}
\]
By Standardization, there is a standard function $\Gamma: B_w/E_w \to \bR$ such that, for standard $f$,  $\Gamma(f/E ) =  \sh(f(w)).$
It is easy to verify that $\Gamma \upharpoonright (B_w/E_w)  \cap\bbS$ is a one-one mapping of $\widehat{M}_w \cap \bbS$ onto 
$\bR \cap \bbS$ that preserves the metrics:
\[
\begin{aligned}
\widehat{D}(f/E, g/E) &= D(f,g) = \sh(|f(w) - g(w)|) \\&= |\sh(f(w)) -
\sh(g(w))| = |\Gamma(f) - \Gamma(g)|.
\end{aligned}
\]
Moreover, $\Gamma((c_x)_E) = x$ for any standard $x \in \bbQ$.
By Transfer, $\Gamma$ is an isometric isomorphism of $\widehat{M}_w$ and $\bR$ which is the identity on $\bbQ$.
In particular, the nonstandard hull is independent of the choice of $w$.
\end{example}

\begin{example}\label{exx2}
Let $M= \bN$ and let $d$ be the discrete metric on $M$; ie,  $d(x,z) = 1$ for all  $x,z \in M$, $x \neq z$.
As all points of $M$ are finite with respect to this metric, we have $B_w = \bN^{\bN}$. Also $$E_w = {}^{\st}\{ \la f, g \ra \in \bN^{\bN}
\times \bN^{\bN} \,\mid\, f(w) = g(w)\}$$ and, for standard $f \in \bN ^{\bN}$, $f/E = 
{}^{\st} \{ g \in \bN^{\bN} \,\mid\, g(w) = f(w)\}  = f_w$.
The space $M$ is identified with a subset of  $\widehat{M}_w$ via $\Gamma:  (c_x)_w \mapsto x$.

In Remark~\ref{ultrapower} 
we establish that 
$\widehat{M}_w =B_w/E_w = \bN^{\bN}/w$ is exactly the ultrapower of $M$ by $U_w$, an ultrafilter over $\bN$ generated by $w$;  in particular, it has the cardinality of the continuum.  Also, $\widehat{D}$ is the discrete metric on $\widehat{M}_w$. By replacing $\bN$ with $I$ as in Remark~\ref{remark6} one can obtain nonstandard hulls of arbitrarily large cardinality.
\end{example}

\begin{example}\label{exx3}
Approachable points play an impotant role in the study of nonstandard hulls.
We define the concept as follows.

Let $(M, d)$ be a standard metric space.  A point $x \in M$ is \emph{approachable} if for every standard $\epsilon > 0$ there is a standard $a \in M$ such that $d(x, a) \le \epsilon$.

The approachable points in $M \cap \bbS_w$ become exactly the standard points of the closure of $M$ in its nonstandard hull $\widehat{M}_w$. Indeed, for standard $f \in B_w$ with $f(w) = x \in M$ we have 
$$x \text{ is approachable } \eqi \forall^{\st} \epsilon > 0 \; \exists^{\st} a \in M\; (d(x,a) \le \epsilon)
\eqi$$
$$\forall^{\st} \epsilon > 0 \; \exists^{\st} a \in M\; (\widehat{D}(f/E,\,a) \le \epsilon) \eqi
\forall \epsilon > 0 \; \exists a \in M\; (\widehat{D}(f/E,\,a) \le \epsilon),$$
 where the last step is by Transfer.
Thus in Example~\ref{exx1} all finite $x \in \bbQ$ are approachable and consequently all standard points in $\bR$ are in the closure of $M = \bbQ$. By Transfer, this is true for all points in $\bR$.
In Example~\ref{exx2} all nonstandard points of $\bN$ are inapproachable and therefore $M$ is closed  in $\widehat{M}_w$. 
\end{example}

\begin{remark}
In this and other constructions in this section we use the coding based on  $\Psi$. The advantage of this choice is that it produces spaces of functions (see in particular Subsection~\ref{inthulls}). 
One can use $\widetilde{\Psi}$ instead, and define eg. 
$\widetilde{B}_w =\, {}^{\st} \{ f_w  \,\mid\, f(w) \text{ is a finite point of }M\}.$
The advantage here is that one gets an  isomorphism 
of the external structure $( \bbB_w, \,\bbE_w,\, \bbD_w)$ with
$( B_w \cap \bbS, \;E_w \cap \bbS, \;D_w \cap \bbS)$. It is thus immediately apparent that 
$((\widetilde{B}_w/ \widetilde{E}_w)\cap \bbS, \,(\widetilde{D}_w/\widetilde{E}_w) \cap \bbS) $ would be isomorphic to 
$(\bbB_w/\bbE_w, \,\bbD_w/\bbE_w)$ if the latter quotient could be formed in $\BST$. (It can be formed in $\HST$ and this claim is a theorem there.)
For the final result the choice of coding method does not matter.

\begin{proposition}
 The  structures 
$(\widetilde{B}_w/ \widetilde{E}_w, \,\widetilde{D}_w/\widetilde{E}_w) $ and 
$(B_w/ E_w, \,D_w/E_w) $ are isomorphic. 
\end{proposition}

\begin{proof}
For standard $f,g \in M^{\bN}$,
$f \in B_w$ iff $f_w \in \widetilde{B}_w$, 
$\la f, g \ra \in E_w $ iff $\la f_w, g_w \ra \in \widetilde{E}_w$ and $D_w(f,g) = \widetilde{D}_w(f_w,g_w)$.
\end{proof}
\end{remark}

\begin{remark}\label{remark6}
In the  construction of $(\widehat{M} , \widehat{D} \,)$ one can replace $\bN$ by any infinite set $I$, as long as $w \in I$ is  good.
If $\bbS_{w_1} \subseteq \bbS_{w_2}$, fix a standard function $h \in I_1^{I_2}$ such that $h(w_2) = w_1$ 
and define the standard mapping $H:  M^{I_1} \to M^{I_2}$ by $H(f) = f \circ h$.
Then $H$ is an embedding of $(B_{w_1},\, E_{w_1},\,D_{w_1})$ into $(B_{w_2},\, E_{w_2},\,D_{w_2})$
in an obvious sense, and it factors to an isometric embedding of $(B_{w_1}/ E_{w_1}, \,D_{w_1}/E_{w_1}) $ 
into  $(B_{w_2}/ E_{w_2}, \,D_{w_2}/E_{w_2}) $.
Given any good $w_1$ and $w_2$, let $w = \la w_1, w_2\ra$; then both 
$(\widehat{M}_{w_1} , \widehat{D}_{w_1} )$ and $(\widehat{M}_{w_2} , \widehat{D}_{w_2} )$
embed into $(\widehat{M}_{w} , \widehat{D}_{w} )$.
\end{remark}

\subsection {Completeness of the nonstandard hull }\label{proofnshull}

This subsection illustrates how one can work with the nonstandard hull as constructed in Subsection~\ref{nshulls}.

\begin{theorem} \label{nshull} $(\widehat{M}_w , \widehat{D}_w \,)$ is a complete metric space.
\end{theorem}

\begin{proof}
Let $\la F_n  \, \mid \, n \in \bN \ra$ be a standard Cauchy sequence in $(\widehat{M} , \widehat{D}\, )$;
we prove that it converges to some  $F \in \widehat{M}$.

Using the Axiom of Countable Choice we obtain a standard sequence $\la f_n  \, \mid \, n \in \bN \ra$ 
such that $F_n = (f_n)_E$ for all $n \in \bN$.

For $k \in \bN$ let $n_k$ be the least element of $\bN$ greater than or equal to $k$ such that
$$\forall m,n \; \left( n_k \le n\le m  \imp \widehat{D}(F_n, F_m) <\tfrac{1}{ k+1}\right) ;   $$
note that the sequence $\la n_k \,\mid\, k \in \bN \ra$ is standard.
From the definition of $\widehat{D}$ we obtain, for standard $k$,
$$\forall^{\st} m,n \; \left( n_k \le n \le m  \imp  d(f_n(w), f_m(w)) < \tfrac{1}{ k+1}\right)   .$$
Hence $\forall^{\st} k \;\exists^{\st} m \; \forall \ell \le k$
$$\left[\, n_{\ell} \le m \,\wedge\, 
\forall n \; \left( n_{\ell} \le n \le m  \imp  d(f_n(w), f_m(w)) < \tfrac{1}{ \ell+1} \right)\right]. $$

By Countable Idealization into $w$-standard sets we get
$$ \exists^{\st_{w}}  m \; \forall^{\st} k\,\left[\, n_{k} \le m \,\wedge\, 
\forall n \; \left( n_{k} \le n \le m  \imp  d(f_n(w), f_m(w)) < \tfrac{1}{ k+1} \right)\right]. $$

Fix such an $m$; clearly it is unlimited. Consider the standard point  $p = f_{n_0}(0)$.
We have $d(p, f_m(w)) \le d(f_{n_0}(0),  f_{n_0}(w) )    + d(f_{n_0}(w), f_m(w))$.
The first term on the right side of the inequality is limited because $f_{n_0} $ is standard and $f_{n_0} \in B$, and the second term is $< 1$ (take $k = 0$).
Hence  $f_m(w)$ is a finite element of $M$.

We note that $f_m(w)$ is $w$-standard; hence
 there is a standard function $f\in M^\bN$ such that $f(w) = f_m(w) $.
It follows that $f \in B$. We let $F = f_E \in \widehat{M}$.

We have that, for all standard $k$,
$$\forall^{\st}  n \; \left( n_{k} \le n   \imp  d(f_n(w), f(w)) < \tfrac{1}{ k+1} \right), $$
and hence 
$$\forall^{\st}  n \; \left( n_{k} \le n   \imp  \widehat{D}(F_n, F) \le \tfrac{1}{ k+1} \right).$$
This shows that the sequence $\la F_n  \, \mid \, n \in \bN \ra$ converges to $F$.
\end{proof}
The proof of Theorem~\ref{proofnshull}  goes through for any infinite set $I$ in place of $\bN$, as long as $w \in I$ is  good.


\subsection
{Nonstandard hulls of standard uniform spaces}\label{uniform}
We generalize the construction of nonstandard hulls in $\BST$ to uniform spaces.

Let $(M, \Delta)$  be a standard uniform space. That is, $M$ is a standard set and $\Delta$ is a standard family of pseudo-metrics on $M$ which endows $M$ with a Hausdorff uniform structure.
In a uniform space, $x \in M$ is \emph{finite} if for all standard $d \in \Delta$, $d(x,p)$ is limited for some (equivalently: for all) standard $p \in M$.
Points $x$ and $  y  $ are \emph{infinitely close} if $d(x,y) \simeq 0$ 
 for all standard $d \in \Delta$.

We fix a standard infinite set $I$ and a $w \in I$ so that  Idealization into $w$-standard sets  over standard sets of cardinality $\le \kappa$ holds for $\kappa = \max\{ |\Delta|, \aleph_0\}$  (see   Proposition~\ref{satuniv}).
A construction of the nonstandard hull of $(M, \Delta)$ can now proceed much as in Subsection~\ref{nshulls}.
We omit the subscripts indicating its dependence on $w$.

We let 
$B  = \, {}^{\st} \{ f \in M^I \,\mid\, f(w) \text{ is a finite point of }M\}$ and\\
$E  = \, {}^{\st} \{  \la f, g \ra \in B \times B \,\mid\, d(f(w), g(w)) \simeq 0 
\text{ for all standard } d \in \Delta \}.$

For each standard $d \in \Delta$  the standard function $D$ on $B \times B$, 
as well as $\widehat{M}$, $\widehat{D}$ and $c$, 
are defined as in Subsection~\ref{nshulls}.
Each $\widehat{D}$ is a pseudo-metric on $\widehat{M}$. 
We let $\widehat{\Delta} =\, {}^{\st} \{\widehat{D}  \,\mid\, d \in\Delta \cap \bbS\}$. 

\begin{theorem}\label{nshulls2}
The structure $(\widehat{M}, \widehat{\Delta})$ is a complete Hausdorff uniform space and 
 $c$ embeds $(M, \Delta)$ into  $(\widehat{M}, \widehat{\Delta})$.
\end{theorem}

\begin{proof}
The proof follows the lines of the proof of Theorem~\ref{nshull}; the main difference is that Cauchy sequences have to be replaced by Cauchy nets.

Let $\la \Lambda, \le \ra$ be a standard directed set and $\la F_{\lambda} \,\mid\, \lambda \in \Lambda \ra$ a standard Cauchy net indexed by $\Lambda$.
Using the Axiom of Choice  one obtains a standard net $\la f_{\lambda} \,\mid\, \lambda \in \Lambda \ra$ 
such that $f_{\lambda} \in F_{\lambda}$ holds for all $\lambda$; then
$\la f_{\lambda}(w) \,\mid\, \lambda \in \Lambda \ra$ is  a $w$-standard net.

The Cauchy property implies that for every standard $d\in \Delta$ there is a standard sequence $\la \lambda_k^d \,\mid \,k \in \bN \ra$
of elements of $\Lambda$ such that $\lambda_k^d \le \lambda_{k+1}^d $ holds for all $k$ and 
$$\forall k\,\forall \lambda, \mu \in \Lambda \; \left( \lambda_k^d \le \lambda \le \mu  \imp \widehat{d}(F_{\lambda}, F_{\mu}) <\tfrac{1}{ k+1}\right) .$$

Hence
$$\forall^{\st} k\,\forall^{\st} \lambda, \mu \in \Lambda \; \left( \lambda_k^d \le \lambda \le \mu  \imp d(f_{\lambda}(w), f_{\mu}(w)) <\tfrac{1}{k+1}\right).$$
We conclude that for every standard $k \in \bN$, every standard finite $\Delta_0 \subseteq \Delta$ and every standard finite $\Lambda_0 \subseteq \Lambda$ 
there is a standard $\mu \in \Lambda$ such that for  all $\ell \le k$, all $d \in \Delta_0$ and all $\lambda \in \Lambda_0$
$$  \lambda_{\ell}^d \le \mu \,\wedge\, 
 \left( \lambda_{\ell}^d \le \lambda    \imp  d(f_{\lambda}(w), f_{\mu}(w)) < \tfrac{1}{\ell+1} \right). $$
By Idealization into $w$-standard sets over sets of cardinality $\le \kappa = \max\{ |\Delta|, \aleph_0\}$ we get a $w$-standard $\mu \in \Lambda$ such that for  all standard $ k$,
all standard $ d \in \Delta$ and all standard $\lambda \in \Lambda$
$$  \lambda_k^d \le \mu \,\wedge\, 
\left( \lambda_k^d \le \lambda    \imp  d(f_{\lambda}(w), f_{\mu}(w)) < \tfrac{1}{k+1} \right). $$

Fix such a $\mu$; as in Subsection~\ref{nshull}
we have that $d(p, f_{\mu}(w))$ is limited for every standard $d \in \Delta$, ie, $f_{\mu}(w)$ is a finite point of $M$.

By Representability there is a standard function $f \in M^I$ such that  $f(w) = f_{\mu}(w)$.
It follows that $f \in B$. Let $F = f/E \in \widehat{M}$.

We have that, for all standard $d \in \Delta$ and $k \in \bN$,
$$\forall^{\st}  \lambda \in \Lambda \; \left( \lambda_{k}^d \le \lambda   \imp  d(f_{\lambda}(w), f(w)) < \tfrac{1}{k+1} \right) , $$
and hence 
$$\forall^{\st}  \lambda \in \Lambda\; \left( \lambda_{k}^d \le \lambda    \imp  \widehat{D}(F_{\lambda}, F) \le \tfrac{1}{k+1}\right).$$
This shows that the net $\la F_{\lambda}  \, \mid \, \lambda \in \Lambda \ra$ converges to $F$.
\end{proof}

\subsection{Internal normed vector spaces} \label{inthulls}

Under many circumstances the type of construction carried out in Subsections~\ref{nshulls},~\ref{uniform}
generalizes to internal structures. We illustrate it in the case of normed vector spaces.


Let $M$ be an internal normed vector space over $\bR$.
This means that $M$, the operations of addition $+$ on $M \times M$ and scalar multiplication $\cdot$  on $\bR \times M$, and the $\bR$-valued norm $ \| . \| $ on $M$, are (internal) sets  and satisfy the usual properties.
In order to be able to apply our coding technique 
we fix a standard set $I$ and a good $w \in I$ so that the set $M$, the above operations and the norm belong to $\bbS_w$. Other parameters relevant to a particular investigation can also be made to belong to $\bbS_w$.
By Transfer in $(\bbV, \bbS_w, \in)$, the properties of these objects that are expressible by $\in$-formulas continue to hold in $\bbS_w$, so we can carry out the desired construction ``over $\bbS_w$'' rather than ``over $\bbV$.''

We first describe the external construction.
Let $$\bbB_w = \pmbaa x \in M \cap\, \bbS_w \,\mid\, \|x\| \text{ is limited } \pmbbb \text{ and }   
\bbE_w = \pmbaa x \in M \cap \,\bbS_w \,\mid\, \|x\| \simeq 0 \pmbbb.$$
It is clear that $\bbB_w$ is an external vector space over the external field $\bR \cap \bbS$ and $\bbE_w$ is its subspace.
Define an external equivalence relation $\approx$ on $\bbB_w$ by $x \approx y \eqi x - y \in \bbE_w \eqi \| x-y \| \simeq 0$.
Obviously, for $x_1, x_2, y_1, y_2 \in \bbB_w$ and $c \in \bR \cap \bbS$ we have
$x_1 \approx y_1 \,\wedge\, x_2 \approx y_2 \imp x_1 + x_2 \approx y_1 + y_2$, 
$c\cdot x_1 \approx c\cdot y_1$ and $\|x_1\| \simeq \|y_1\|$.
If external collections of external sets were available in $\BST$, one could now form the quotient space $\bbB_w/\bbE_w$ with the norm $\|x_{\bbE_w}\| = \sh(\|x\|)$, which would then be an external normed vector space over the external field $\bR \cap \bbS$.
This construction is of course not possible in $\BST$ directly, but the coding introduced in Subsection~\ref{coding} enables us to carry it out and produce a standard normed metric space which, when viewed from the standard point of view, is (externally) isomorphic to $\bbB_w/\bbE_w$.

It follows from  Representability that there is a standard function $\la (M_i, \,+_i, \, \cdot_i) \,\mid\, i \in I \ra$ such that $(M_w, \,+_w, \, \cdot_w) = (M, \,+, \, \cdot)$
and, for all $i \in I$,  $(M_i, \,+_i, \, \cdot_i)$ is a vector space over $\bR$.
Let 
$$B_w = {}^{\st} \{f  \in \Pi_{i \in I} M_i \,\mid\, \|f(w)\| \text{ is limited} \} ;$$
$$E_w =  {}^{\st} \{f  \in \Pi_{i \in I} M_i \,\mid\, \|f(w)\| \simeq 0\}.$$
Addition and scalar multiplication on $B_w$ are defined pointwise: 
$$ (f+ g)(i) = f(i) +_i g(i) \text{\qquad and \qquad} (c \cdot  f)(i) = c \cdot_i f(i)$$
for $f, g \in B_w$, $c \in \bR$, and  all $i \in I$.
By Standardization, there is a standard function $\|.\|$ such that for all standard $f \in B_w$ we have 
$\|f\| = \sh(\|f(w)\|)$.

It is routine to verify that $B_w$ is a standard vector space over $\bR$, $\|.\|$ is a pseudo-norm on $B_w$, and $f \in E_w \eqi \|f\| = 0$.
The quotient $\widehat{E}_w = B_w/E_w$ is thus a well-defined standard normed vector space. 
The proof given in Subsection~\ref{proofnshull}  shows that $\widehat{E}_w = B_w/E_w$
is  complete.


\subsection
{Loeb measures} \label{loeb}

Let $(\Omega, \cA, \mu)$ be an internal finitely-additive measure space, with $\Omega \subseteq O$ for a standard set $O$. As discussed in Subsection~\ref{notloeb},  attempts to construct its Loeb extension
``over $\bbV$''  are only partially successful in $\BST$.
We fix a standard set $I$ and a good $w\in I$ so that $\Omega, \cA , \mu $ are $w$-standard.
By Transfer, $(\Omega, \cA, \mu)$ is an internal finitely-additive measure space in the sense of $\bbS_w$, and we construct the Loeb extension ``over $\bbS_w$.'' 
We usually do not  indicate the dependence on the choice of $w$.
In this example it is convenient to employ the variant of coding from Definition~\ref{coding2}.

For every  $X \in \cA\cap \bbS_w$ let 
\begin{equation}
[X] =\, {}^{\st}\{ f_w \,\mid\, f \in O^I \,\wedge\, f(w) \in X \cap \bbS_w \}.
\end{equation}
If $X, X_1, X_2 \in \cA\,\cap \,\bbS_w$, then the equivalences 
$ f_w \in [X_1 \cap X_2] $ iff $f_w \in [X_1]\,\wedge\, f_w \in [X_2]$ and 
$f_w \in [\Omega \setminus X]$ iff $f_w\in [\Omega] \,\wedge\,f_w \notin [X]$
hold for all standard $f \in O^I$.
By Transfer, $  [X_1 \cap X_2] = [X_1]\cap [X_2]$ and 
$ [\Omega \setminus X]= [\Omega] \setminus [X]$.
Let 
$$\ccB =\, {}^{\st} \{ A \in \cP(O^I)\,\,\mid\,  A = [X] \text{ for some }  X \in \cA \cap \bbS_w\} .$$

Using Transfer again, it follows that $\ccB$ is a standard algebra of subsets of $[\Omega]$.
We note that $X_1, X_2 \in \cA \cap \bbS_w$, $X_1 \ne X_2$, implies $[X_1] \ne [X_2]$, so for standard $A \in \ccB$ there is a unique $X \in \cA \cap \bbS_w$ with $A = [X]$.

A standard finitely additive measure $m$ on the algebra $\ccB$ with values in the interval $[0, +\infty]$ is determined by the requirement that for standard $A = [X] \in \ccB$
$$m(A) = \sh(\mu(X))  .$$

The measure space $([\Omega], \ccB, m)$ satisfies  Carath\'{e}odory's condition.

\begin{lemma}
If $\la A_k \,\mid\, k \in \bN\ra$ is a sequence of mutually disjoint  sets in $\ccB$, $A \in \ccB$, and 
$A = \bigcup_{k \in \bN}\, A_k$, then $m(A) = \Sigma_{k \in \bN}\; m(A_k)$.
\end{lemma}

\begin{proof}
In view of Transfer, it suffices to prove this claim under the assumption that $\la A_k \,\mid\, k \in \bN\ra$ and $A$ are standard.

Suppose $A = [X]$ where $X  \in \cA \cap \bbS_w$ and for each standard $k$, $A_k = [X_k]$ where $X_k   \in \cA \cap \bbS_w$.
Clearly $X_k $ are mutually disjoint  and $X_k  \subseteq X$  for all standard $k$.

Assume that  for every standard $n$ there is a $w$-standard  $x \in X$ such that $x \notin X_k$  holds for all $k \le n$.
By Countable Idealization into $\bbS_w$ there is a $w$-standard  $x \in X$ such that $x \notin X_k$  holds for all standard $k$.
By Representability, $x = g(w)$ for some standard $g  \in O^I$. 
Then $g_w \in A$ but $\forall^{\st} k \, (g_w \notin A_k)$ and,
by Transfer, $\forall k \, (g_w \notin A_k)$.
This is a contradiction.

Therefore there is a standard $n$ such that $\forall^{\st_w} x \in X
\, \exists k \le n \, (x \in X_k)$.  It follows that $\forall^{\st} F
\in A \, \exists k \le n \, (F \in A_k)$.  By Transfer, $ \bigcup_{k
  \in \bN} \, A_k= \bigcup_{k \le n} \, A_k$ and, by finite additivity
of $m$, we obtain $m(A) = \Sigma_{k \le n}\; m(A_k) = \Sigma_{k \in
  \bN}\; m(A_k)$.
\end{proof}

We conclude that $m$ can be extended to a $\sigma$-additive measure $\overline{m}$ 
with values in $[0, +\infty]$
on the $\sigma$-algebra  $\overline{\ccB}$   generated by $\ccB$. The measure-theoretic completion of $([\Omega],  \overline{\ccB}, \overline{m})$ 
 is the desired Loeb measure space; 
we denote it  $([\Omega], \widehat{\ccB}, \widehat{m})$ 
(of course, it depends on the choice of $w$).
Instead of an appeal to the Carath\'{e}odory's theorem, a direct proof along the lines of~\cite{H3} can be given; see also Albeverio et al.~\cite{A}, Remark 3.1.5 and the references therein.

\begin{remark}
In order to explicate the relationship of $([\Omega], \widehat{\ccB}, \widehat{m})$ to the usual external Loeb  measure space, we first recall that $\bbS_w$ satisfies the statement that
$(\Omega, \cA, \mu)$ is a finitely-additive measure space.
The Loeb construction carried out over $\bbS_w$ would start with the external finitely additive measure space $(\boldsymbol{\Omega}, \boldsymbol{\cA}, \mathbf{m} )$, where $\boldsymbol{\Omega} = \Omega \cap \bbS_w$, $\boldsymbol{\cA} = \pmbaa X \cap \bbS_w \,\mid\, X  \in \cA \cap \bbS_w\pmbbb$, and  $ \mathbf{m}(X \cap \bbS_w) = \sh(\mu(X))$
for $X \in \boldsymbol{\cA}$.
This space would be extended to an external $\sigma$-additive measure space using the Carath\'{e}odory's theorem, and then completed. 
We now take $V = \Omega \cup \bR$, say, and note that 
for $X \in \cA \cap \bbS_w$,
$[X]$ is a $w$-code of $X \cap \bbS_w$, and hence  $\ccB$ is a $w$-code for $\boldsymbol{\cA}$.
It follows that  $([\Omega],  \ccB, m)$ is a $w$-code for $(\boldsymbol{\Omega}, \boldsymbol{\cA}, \mathbf{m} )$. The above construction cannot be carried out in $\BST$ directly for the  external measure space $(\boldsymbol{\Omega}, \boldsymbol{\cA}, \mathbf{m} )$, but presents no difficulties for its  $w$-code $([\Omega],  \ccB, m)$, a standard finitely-additive measure space.
\end{remark}

\begin{remark}
We compare  Loeb measure spaces obtained from the same $(\Omega, \cA, \mu)$ for different choices of the parameter $w$.
We use subscripts to indicate dependence on this parameter and fix good
 $w \in I$, $z \in J$ where $I,J$ are standard and $\bbS_w \subseteq \bbS_z$.

\begin{proposition} \label{embedding}
There is a standard embedding $\widehat{H} = \widehat{H}_{w,z}$ of $\widehat{\ccB}_w$ into  $\widehat{\ccB}_z$ that  preserves  complements and countable unions and restricts to an isomorphism of $\ccB_w$ and $\ccB_z$. 
If $m_w$ is $\sigma$-finite, then also $\widehat{m}_w (B) = \widehat{m}_z (\widehat{H} (B) )$ for all $B \in \widehat{\ccB}_w$.
\end{proposition}

\begin{proof}
By Standardization, there is a standard function $\widehat{H}: \widehat{\ccB}_w \to \widehat{\ccB}_z$ such that
$\widehat{H} (B) =  \widetilde{\Psi}_z (\widetilde{\boldsymbol{\Psi}}^{-1}_w (B))$
for standard $B \in \widehat{\ccB}_w$.
For$A \in \ccB_w \cap\, \bbS$, if $A  = [X]_w$ for $X \in \cA \cap \bbS_w$ then $\widehat{H} (A) = [X]_z \in \ccB_z\cap \bbS$, and vice versa. This shows that $\widehat{H} $ maps $\ccB_w \cap \bbS$ onto $\ccB_z \cap \bbS$, and hence, by Transfer, $\ccB_w$ onto $\ccB_z$.

 In Subsection~\ref{coding}  we point out that the coding $\widetilde{\Psi}$ preserves complements and countable unions, so the same holds for $\widehat{H}$.
We give some details for the countable unions.
Let $B = \bigcup_{n \in \bN } B_n$, where $\la B_n \,\mid\,n \in \bN\ra$ is standard and
$B, B_n \in \widehat{\ccB}_w$ for all $n \in \bN$.
Define the external set $\bbX =\pmbaa \la n, x \ra \,\mid\, x \in \widetilde{\boldsymbol{\Psi}}^{-1}_w (B_n)\pmbbb \subseteq 
(\bN \cap \bbS) \times O$, so that $B_n = \widetilde{\Psi}_w (\bbX_n)$.
We then have 
$B = \widetilde{\Psi}_w(\bigcup_{n \in \bN \cap \bbS} \bbX_n) $, so 
$ \widetilde{\boldsymbol{\Psi}}^{-1}_w (B) =\bigcup_{n \in \bN \cap \bbS} \bbX_n $ and 
$\widehat{H} (B) =  \widetilde{\Psi}_z (\widetilde{\boldsymbol{\Psi}}^{-1}_w (B) )=  
\widetilde{\Psi}_z (\bigcup_{n \in \bN \cap \bbS} \bbX_n ) =
 \bigcup_{n \in \bN} \widetilde{\Psi}_z(\widetilde{\boldsymbol{\Psi}}^{-1}_w (B_n))  =
 \bigcup_{n \in \bN} \widehat{H} (B_n)$.

It is clear from the definition of $m(A)$ that $m_w(A) = m_z(\widehat{H}(A))$ holds for standard $A \in \ccB_w$, and hence by Transfer, for all $A \in \ccB_w$.
If $m_w$ is $\sigma$-finite, then  the completed Carath\'{e}odory's measure $\widehat{m}_w$ is uniquely determined. 
If $\widehat{m}_w (B) \neq \widehat{m}_z (\widehat{H} (B) )$ for some $B \in \widehat{\ccB}_w$, then 
$\widehat{m}_w $ and $\widehat{m}_z \circ\widehat{H}$ would be two distinct extensions of $m_w$ 
 from $\ccB_w$ to $\widehat{\ccB}_w $, a contradiction.
\end{proof}
\end{remark}

\subsection
{Lebesgue measure from Loeb measure}\label{lebesgue}

As is well known, the Lebesgue measure  can be obtained from a suitable Loeb measure.%
\footnote{\label{f6}A nonstandard construction of the Lebesgue measure can also be carried out  without using a Loeb measure as an intermediate step.  
In the axiomatic approach, one method for doing so is developed by  Lyantse and Kudryk~\cite{LK}, Appendix A, in the framework of $\IST$.
Another way is implicit in Hrbacek~\cite{H3} and explicitly presented  in  Hrbacek and Katz~\cite{HK} in the framework of $\SCOT$, a subtheory of $\IST$ and $\BST$ that conservatively extends $\ZF$ + Dependent Choice.
For a ``radically elementary'' approach see Cartier and Perrin~\cite{CP1, CP2}.
} 
We give the gist of the argument  in our framework, for the interval $[0, 1]$.
More details can be found in Albeverio at al.~\cite{A} (using a model-theoretic approach).

For $n \in \bN$ let $\cT_n = \{ i /n \,\mid\, 0 \le i \le n\}$.
Fix a nonstandard integer $w \in \bN$, and let $\cT = \cT_{w}$
(in  model-theoretic frameworks the equivalent of $\cT$  is called ``hyperfinite time line'').
Let $O = [0,1]$, $\Omega = \cT$, $\cA = \cP(\cT)$ and $\mu$  the counting measure on $\cT$, ie, $\mu(X) = |X|/ |\cT|$ for all sets $X \subseteq \cT$.
The construction in the preceding subsection, with $I =\bN$,  yields the Loeb measure 
$([\Omega], \widehat{\ccB}, \widehat{m})$.

For every standard $A \subseteq [0,1]$ define
$${}^{\bullet}\!\!A  =\, {}^{\st}\{ f _w\in [\Omega] \,\mid\, f(w) \simeq c \text{ for some standard }  c \in A\}.$$
 This just means that ${}^{\bullet}\!\!A$ is a $w$-code for 
$  \sh^{-1} (A) \cap \cT \cap \bbS_w  $.
Standard elements of the set ${}^{\bullet}\!\!A$ are those  $f_w \in [\Omega]$ whose value ``at infinity'' (ie, at $w$) is infinitely close to a standard real  in $A$.

Let $\cL =\, {}^{\st}\{ A \subseteq [0,1] \,\mid\,{}^{\bullet}\!\!A \in  \widehat{\ccB} \,\}$
and let $\ell$ be the standard function on $\cL$ determined by the requirement that 
$\ell(A) = \widehat{m}({}^{\bullet}\!\!A) $ for all standard $A \in \cL$.

\begin{theorem}
The triple $([0,1], \cL, \ell)$ is the Lebesgue measure space on $[0,1]$.
\end{theorem}

\begin{proof}
We prove that $\cL$ is a $\sigma$-algebra containing all standard open intervals $(a,b) \subseteq [0,1]$ and all singletons $\{c \}$ for standard $c \in [a,b]$. We also prove that $\ell$ is $\sigma$-additive and   $\ell( (a,b)) = b-a$,  $\ell(\{c\}) =0$ for all standard open intervals and singletons, respectively. 
This implies that $\cL$ contains all Lebesgue measurable subsets of $[0,1]$ and that $\ell$ is the Lebesgue measure for such subsets. For a proof (in the model-theoretic  framework) that all sets in $\cL$ are Lebesgue measurable see~\cite{A}, Proposition 3.2.5; it can be easily adapted to our framework.

Let $\la A_k \,\mid\, k \in \bN\ra$ be a standard sequence of elements of $\cL$
and let $A = \bigcup_{k \in \bN} A_k$.
 Then
${}^{\bullet}\!\!A_k\in \widehat{\ccB} $ holds for all standard $k \in \bN$, and
we obtain that ${}^{\bullet}\!\!A = \bigcup_{k \in \bN} {}^{\bullet}\!\!A_k$, because 
if $f(w) \simeq a $ for some standard $a \in A$, then $a \in A_k$ for some standard $k \in \bN$ by Transfer. As $\widehat{\ccB}$ is a $\sigma$-algebra, ${}^{\bullet}\!\!A \in \widehat{\ccB}$ and we conclude that 
$A \in \cL$.
Furthermore, $\ell(A) =\widehat{m}({}^{\bullet}\!\!A)=  \Sigma_{k \in \bN}\; \widehat{m}({}^{\bullet}\!\!A_k)
= \Sigma_{k \in \bN}\;\ell(A_k)$, establishing $\sigma$-additivity of $\ell$.

Given a standard open interval $(a, b) \subseteq [0,1]$, we let
$X_{a,b}  = \cT \cap (a,b)$.
We have $[X_{a,b}] \in \ccB$ and $m([X_{a,b}]) = \sh(\mu( X_{a,b} ) )= b-a$. 

Let $A =  (a,b) $; it remains to observe that, for standard $f_w \in [\Omega]$, $f_w \in {}^{\bullet}\!\!A$ iff 
$f_w \in  [X_{a + 1/m, \;b - 1/m} ]$ for some standard $m \in \bN \setminus \{0\}$.
Standardization gives the function $\la A_m\,\mid\, m \in \bN\setminus \{0\}\ra$
where $A_m = [X_{a + 1/m, \;b - 1/m} ]$ for standard $m$, 
and Transfer implies 
${}^{\bullet}\!\!A = \bigcup_{m \in \bN\setminus \{0\}} A_m$.
It follows that ${}^{\bullet}\!\!A \in \overline{\ccB}$ because the latter is a $\sigma$-algebra, 
and that $\overline{m}(\overline{A}) = b - a$ because $\overline{m}$ is a $\sigma$-additive extension of $m$.
Hence $A \in \cL$ and $\ell(A) = b - a$.

The argument for $A = \{c\}$ is similar, using the fact that ${}^{\bullet}\!\!A = \bigcap_{m \in \bN\setminus \{0\}} A_m$ with $A_m = [X_{c- 1/m,\, c + 1/m} ]$ for standard $m$.
\end{proof}


\subsection{Neutrices and external numbers}
This is another application of nonstandard analysis that extensively uses external sets; see Dinis and van den Berg~\cite{DvdB}. 
A \emph{neutrix} is a convex additive subgroup of $\bR$.
With the exception of $\{0\}$ and $\bR$, neutrices are externals sets; the monad $\bbM(0)$ and the galaxy $\bbG(0)$ are nontrivial examples.
An \emph{external number} is an algebraic sum of a real number and a neutrix.
Addition and multiplication of external numbers are defined by the Minkowski operations.
Let $\mathcal{N}$ denote the collection of all neutrices and
$\mathcal{E}$ denote the collection of all external numbers.
Even in $\HST$, $\mathcal{N}$ and $\mathcal{E}$ are (definable) proper classes of external sets; they are ``too large'' to be external sets  (see Kanovei and Reeken~\cite{KR2}).

This difficulty can be remedied by relativizing these concepts to  $\bbS_w$ (for good $w$). 
In fact, $\mathcal{N}_w =
\pmbaa \mathbb{X} \cap \bbS_w\,\mid\, \mathbb{X}  \in \mathcal{N} \pmbbb$ and 
$\mathcal{E}_w =
\pmbaa a + \mathbb{X} \cap \bbS_w\,\mid\, a \in \bR \cap \bbS_w \,\wedge\,\mathbb{X}  \in \mathcal{N} \pmbbb$.
The collections $\mathcal{N}_w$ and $\mathcal{E}_w $ are external sets (of external sets) in $\HST$.

As the natural embedding of $\bR \cap \bbS_w$ into $\mathcal{E}_w$ given by $r \mapsto r+\{0\}$ is crucial for applications of these concepts, it may be best for most purposes to avoid coding as much as possible. 
For example, the study of algebraic properties of operations $+$ and $\times$  on  $\mathcal{E}_w$ can be carried out in $\BST$ while viewing external numbers as  
external subsets of  $\bR \cap \bbS_w$ .
However, work with external numbers often focuses on the structure $(\mathcal{E}_w , +, \times)$, its subsets, functions with values in it, and so on. 
Then one can use the  techniques of Subsection~\ref{coding}; see in particular Definition~\ref{coding2} with $V = \bR$, to code $\mathcal{N}_w$ and $\mathcal{E}_w$ by standard structures.

First, the external set $\bR \cap \bbS_w$ is coded by the standard set $\bR^I/w$. As shown in Remark~\ref{ultrapower},  $\bR^I/w= \bR^I/U_w = {}^{\ast}\bR$, the hyperreals constructed as the standard ultrapower of $\bR$ by the standard ultrafilter $U_w$ generated by $w$. 
Let ${}^{\ast}\!\!< $ , ${}^{\ast}\!+ $ and ${}^{\ast}\!\times$ be the ordering, addition and multiplication on the hyperreals.
The coding provides an external isomorphism between   $({}^{\ast}\bR \cap \bbS, {}^{\ast}\!\!<, {}^{\ast}\!+, {}^{\ast}\!\times) $ and $(\bR \cap \bbS_w ,<, +, \times) $, so we can informally identify ${}^{\ast}\bR \cap \bbS$ and $\bR \cap \bbS_w$.
 Neutrices and external numbers in the hyperreals 
$({}^{\ast}\bR, {}^{\ast}\!\!<, {}^{\ast}\!+, {}^{\ast}\!\times) $
can be defined the same way as in  $\bR$. External numbers 
in $\bR \cap \bbS_w$ are coded by the standard  external numbers  in the hyperreals ${}^{\!\ast} \bR$. The coding preserves algebraic operations on external numbers.
The collections $\mathcal{N}_w$ and $\mathcal{E}_w$ are coded   respectively by
the standard sets $N$ and $E$ of all neutrices and external numbers in  ${}^{\!\ast} \bR$.

This approach is admittedly rather awkward.
In $\HST$ the universes $\bbS_w$ can  be extended to ``external universes'' $\mathbf{WF}(\bbS_w)$
(see Kanovei and Reeken~\cite{KR}, Sections 6.3, 6.4 for details).
Perhaps the most practical way to handle external numbers  would be to work with $\mathcal{N}_w$ and $\mathcal{E}_w$  in these universes 
and in the end code the final results in $\BST$, if desired.


\section{Subuniverses and ultrapowers}  \label{appB}

\subsection{Subuniverses $\bbS_w$ and ultrapowers}\label{universes}

The universe $\bbS_w$ is closely connected to the ultrapower of the standard universe by a standard ultrafilter.  

The principle of Standardization implies that there is a standard set $U_w$ such that 
\begin{equation}
\forall^{\st} X \, ( X \in U_w \eqi  X \in \cP(I) \,\wedge\,  w \in X).
\end{equation}
Clearly 
(i) $\emptyset \notin U_w$,
 and for standard $X, Y \in \cP(I)$

(ii) $X \in U_w \,\wedge\, X \subseteq Y \imp Y \in U_w$;

(iii) $ X, Y \in U_w \imp X \cap Y \in U_w$;

(iv) $ X \in U_w \,\vee\, (I\setminus X) \in U_w$.

By Transfer, (ii) -- (iv) hold for all $X, Y \in \cP(I)$, so $U_w$ is an ultrafilter over $I$. 
One sees easily that $U_w$ is nonprincipal if and only if $w$ is nonstandard. 
Conversely,  Bounded Idealization of $\BST$ implies that for every standard ultrafilter $U$ over $I$ there are (many) $w \in I$  such that $U = U_w$.

The ultrapower of the universe of all sets $\bbV$  by a standard ultrafilter $U$ is defined  in the usual way. 
One defines an equivalence relation $=_U$ on $ \bbV^I$ by 
\begin{equation} f =_U g \eqi \{ i \in I \,\mid\, f(i) = g(i) \} \in U,
\end{equation}
and a membership relation
\begin{equation}
f  \in_U g \eqi \{ i \in I \,\mid\, f(i) \in g(i) \} \in U.
\end{equation}
The usual procedure at this point is to form equivalence classes $f_U$ of functions  $f \in \bbV^I$ modulo $=_U$, using ``Scott's trick'' of taking only the functions of the minimal von Neumann rank to guarantee that the equivalence classes  are sets: Let 
$$ f_U = \{g \in \bbV^I\,\mid\, g =_U f \text{ and } \forall h\in \bbV^I\,
(h =_U f \imp \rank h \ge \rank g)\};$$
see Jech~\cite{J}, (9.3) and (28.15). One lets 
$\bbV^I\!/U $ be the class of all $ f_U$ for $ f \in \bbV^I$
 and defines
\begin{equation}
f_U \in_U g_U \eqi f \in_U g.
\end{equation}
The ultrapower of $\bbV$ by $U$ is the structure $(\bbV^I\!/U,\, \in_U)$.
The universe $\bbV$ is embedded into $\bbV^I\!/U$ via $x \mapsto (c_x)_{U}$
where $c_x$ is the constant function on $I$ with value $x$.

We note that  $f_U$ is standard iff $f_U = g_U $ for some standard $g  \in \bbV^I$.
\emph{We assume from now on that whenever the equivalence class $f_U$ is standard, the representative function $f$ is taken to be standard.}

The key insight is that the \emph{standard} elements of the ultrapower of $\bbV$ by $U_w$
are in equality-and-membership-preserving correspondence with $w$-\emph{standard}
elements of $\bbV$. It is expressed by the following proposition, which is an immediate consequence of definitions (3) - (5).

\begin{proposition}\label{isomorphism}
For any standard functions $f,$ $ g \in \bbV^I$:
$$ f  =_{U_w} g   \eqi f_{U_w} = g_{U_w}  \eqi f(w) = g(w) \quad \text{and}$$
$$ f  \in_{U_w} g    \eqi f_{U_w} \in_{U_w} g_{U_w}   \eqi  f(w) \in g(w). $$
\end{proposition}

The  correspondence $\boldsymbol{\Phi}_w$ is defined on $\bbS_w \times (\bbV ^{I}\!/U_w \,\cap\, \bbS)$ by
$$\boldsymbol{\Phi}_w( \xi, f_{U_w}) \eqi  f(w)  = \xi.$$
In this notation, Proposition~\ref{isomorphism} asserts the following.

\begin{corollary} \label{isom1}
The class $\boldsymbol{\Phi}_w$ is an \emph{isomorphism}
between the structures $(\bbS_w, \,  \in)$ and  $(\bbV^I\!/U_w \,\cap\, \bbS , \, \in_{U_w} )$.
\end{corollary}

We note that $(\bbV^I\!/U_w \,\cap\, \bbS , \, \in_{U_w} )$ is the ultrapower of the universe in the sense of the standard universe $\bbS$.  If $\phi(v) $ is an $\in$-formula such that 
$\bbV^I\!/U_w = \pmbaa F \,\mid\, \phi(F) \pmbbb$, then 
$\bbV^I\!/U_w \,\cap\,\bbS= \pmbaa F \in \bbS\,\mid\, \phi^{\st}(F) \pmbbb$.
If $\psi(u,v)$ is an $\in$-formula such that $F \in_{U_w} G \eqi \psi(F,G)$, then $F \in_{U_w} G \eqi \psi^{\st} (F,G)$ holds for $F,G \in \bbS$. 

If $\boldsymbol{\Phi}_w (\xi, f_{U_w})$ holds, we write $\boldsymbol{\Phi}_w(\xi) = f_{U_w}$.
We note that for $x \in \bbS$, $\boldsymbol{\Phi}_w(x) = (c_x)_{U_w}$ where $c_x$ is the constant function on $I$ with value $x$. 
As is customary, we identify $(c_x)_{U_w}$ with $x$.
This gives a stronger version of Corollary~\ref{isom1}.

\begin{corollary}\label{isom2}
The class $\boldsymbol{\Phi}_w$ is an isomorphism
between the structures $(\bbS_w, \,\bbS, \, \in, )$ and  $(\bbV^I\!/U_w \,\cap\, \bbS , \, \bbS, \, \in_{U_w}  )$.
\end{corollary}

We also note that  $\boldsymbol{\Phi}_w(w) = Id_{U_w}$ where $Id(i) = i $ for all $i \in I$. 

Recall that quantifiers with superscript $\st_w$ range over $w$-standard sets. If  $\phi$ is an $\in$-formula,  $\phi^{\st_w}$ is the formula obtained from $\phi$ by relativizing all quantifiers to $\st_w$. 
 \L o\'{s}'s Theorem for $\in$-formulas  takes the following form:
\begin{lemma} \label{los}
For standard $f_1,\ldots, f_r$,
$$\phi^{\st_w} (f_1(w), \ldots,  f_r (w)) \eqi
\{ i \in I \,\mid\,  \phi( f_1(i),\ldots, f_r(i)) \} \in U_w.$$
\end{lemma}

The structures $(\bbV^I\!/U_w , \, \in_{U_w} )$  and $(\bbV^I\!/U_w \,\cap\, \bbS , \, \in_{U_w} )$ are not models in the sense of model theory because their components are proper classes; hence the satisfaction relation $\vDash$ is not available. Given an $\in$-formula $\phi$ with parameters from $\bbV^I\!/U_w $,  we write ``$\phi$ \emph{holds in} $(\bbV^I\!/U_w , \, \in_{U_w })$'' to stand for the \emph{formula} obtained from  $\phi$ by replacing all occurrences of $u \in v$ with $ u\in_{U_w}\! v$  and relativizing all quantifiers to $\bbV^I\!/U_w $; similarly for 
``$\phi$ \emph{holds in} $(\bbV^I\!/U_w\,\cap\, \bbS  , \, \in_{U_w} )$.''

\bigskip
\emph{Proof of Lemma}~\ref{los}.
We have 
$$ \phi^{\st_w} (f_1(w), \ldots,  f_r (w)) \eqi $$
$$ \phi((f_1)_{U_w},\ldots,(f_r)_{U_w}) \text{ holds in } (\bbV^I\!/U_w \cap \bbS , \, \in_{U_w} )   \text{\quad [by Corollary~\ref{isom1}] }$$
$$\eqi \phi((f_1)_{U_w},\ldots,(f_r)_{U_w}) \text{ holds in } (\bbV^I\!/U_w, \, \in_{U_w} )  \text{\quad [by Transfer]}$$
$$\eqi \{ i \in I \,\mid\,  \phi( f_1(i),\ldots, f_r(i)) \} \in U_w \text{\quad [by  the usual \L o\'{s}'s Theorem]. }$$
\qed

\begin{remark} \label{ultrapower}
In Subsection~\ref{coding} we fix a universal standard set $V$ and for standard $f \in V^I$ define $f_{w,V}$
(see Definition~\ref{coding2}). Clearly 
\[
\begin{aligned}
f_{w,V} &= {}^{\st}\{ g\in V^I \,\mid\, g(w) = f(w)\} = 
 {}^{\st}\{ g\in V^I \,\mid\, g  =_{U_w} f\} 
\\&=\{ g\in V^I \,\mid\, g  =_{U_w} f\},
\end{aligned}
\]
where the last step is by Transfer.
Thus, again by Transfer, $V^I/w$ is nothing but the ultrapower $V^I/U_w$. 
\end{remark}

\subsection{Proofs of claims in Subsection~\ref{wstandard}} \label{universproofs}

\begin{proposition}\label{cinc}
A set $w$ is good iff $U_w$ is countably incomplete. 
\end{proposition}

Of course, ultrapowers by countably incomplete ultrafilters are the ones  of interest in nonstandard analysis.

\begin{proof}
Immediate from  the isomorphism  $\boldsymbol{\Phi}$ between the universe of $w$-internal sets $(\bbS_w, \in)$ and the ultrapower  $((\bbV^I/U_w) \cap \bbS, \,\in_{U_w})$
(see Chang-Keisler~\cite{CK}, Sec. 4.3).
\end{proof}

\emph{Proof of Proposition}~\ref{univ}

\emph{Proof of} (1):
Assume $\exists x\; \phi(x, p_0)$ where  $p_0$ is (wlog. the only) parameter and $\st_w (p_0)$.
Fix a standard set $P$ such that $p_0 \in P$ (Boundedness).
Use $\AC$ to obtain a standard function $F$  on $P$ such that 
$\forall p \in P\, ( \exists x\, \phi(x,p) \imp \phi( F(p), p) ).$
Hence $\phi(F(p_0), p_0)$ holds. As $F(p_0) $ is $w$-standard by Proposition~\ref{propos} (c), we conclude 
that $\exists^{\st_w} x\; \phi(x, p_0)$.\qed

\bigskip

\emph{Proof of} (2):
It is well-known that every ultrapower by a countably incomplete ultrafilter $U$ is $\omega_1$-saturated
(see Chang-Keisler~\cite{CK}, Theorem 6.1.1). Hence Countable Idealization in the form
\begin{equation} \forall^{\st}  n \in \bN\;\exists x\; \forall m \in \bN\; (m \le n \;\imp  \phi(m,x))\eqi 
\exists x \; \forall^{\st} n \in \bN \; \phi(n,x)\end{equation}
holds in $(\bbV^I\!/U_w , \, \bbV, \, \in_{U_w} )$.
By $\BST$ Transfer, (6) holds in $(\bbV^I\!/U_w \cap \bbS, \, \bbS, \, \in_{U_w} )$, and,
by Corollary~\ref{isom2}, it  holds in $(\bbS_w, \bbS, \in)$. This translates to 
$\forall^{\st}  n \in \bN\;\exists^{\st_w} x\; \forall^{\st_w} m \in \bN\; (m \le n \;\imp  \phi^{\st_w}(m,x))$
$\eqi  
\exists^{\st_w} x \; \forall^{\st} n \in \bN \; \phi^{\st_w}(n,x).$
Using  $w$-Transfer we  get the desired form
$$\forall^{\st}  n \in \bN\;\exists x\; \forall m \in \bN\; (m \le n \;\imp  \phi(m,x))\eqi 
\exists^{\st_w} x \; \forall^{\st} n \in \bN \; \phi(n,x).\qed$$

\bigskip
\emph{Proof of} (3):
Recall that
$\st_w(x)$ means that $x = g(w)$ for some standard $g$ defined on $I$. 
Let $\phi(v)$ be an $\in$-formula with standard parameters. 
Assume $\phi(g(w))$; by $w$-Transfer then  $\phi^{\st_w}(g(w))$.
From the isomorphism between $(\bbS_w, \in)$ and  $(\bbV^I/U_w) \cap \bbS, \,\in_{U_w})$  
and \L o\'{s}'s Theorem it follows  that 
$ X = \{ i \in I \,\mid\, \phi(g(i))\} \in U_w$. Pick $i_0 \in X$  and
let $f$ be the standard function defined by 
$$f(i) = g(i) \text{ for  } i \in X; \;f(i) = f(i_0) \text{ otherwise.}$$
Then $f(w) = x$ and $\phi(f(i))$ holds for all $i \in I$.\qed
\qed

\bigskip

\begin{definition}\label{kappasat}
Let $\kappa$ be a standard  infinite cardinal.
The set $w$ is $\kappa^+$-\emph{good} if $U_w$ is a countably incomplete $\kappa^+$-good ultrafilter (In particular, $w$ is \emph{good} iff it is $\omega_1$-good; see Proposition~\ref{cinc}.)
\end{definition}

We prove Proposition~\ref{satuniv} in the following form.

\begin{proposition}\label{satunivcopy}
\emph{(Idealization into $w$-standard sets over sets of cardinality $\le \kappa$)}
Let $\phi$ be an $\in$-formula with $w$-standard parameters. If $w$ is $\kappa^+$-good, then
for every standard set $A$ of cardinality 
$\le \kappa$
$$\forall^{\st\! \fin} a \subseteq A \;\exists y\; \forall x \in a\;
\phi(x,y) \eqi \exists^{\st_w} y \; \forall^{\st} x \in A \; \phi(x,y).$$
 For every standard uncountable cardinal $\kappa$ and every $z$ there exist $\kappa^+$-good
 $w$ so that $z$ is $w$-standard.
\end{proposition}

\emph{Proof of Proposition~\ref{satunivcopy}}

It is well known (Chang and Keisler~\cite{CK},  Theorem 6.1.8) that any ultrapower  by a countably incomplete $\kappa^+$-good ultrafilter $U$ is $\kappa^+$-saturated.
As in the proof of~(2), it follows that  
Bounded Idealization over sets of cardinality $\le \kappa$ holds in $(\bbS_w, \bbS, \in)$ for  $\kappa^+$-good $w$.

To also obtain $z \in \bbS_w$ we use the following fact proved in Keisler~\cite{K}:

If $U$ is a countably incomplete $\kappa^+$-good ultrafilter over $I$ and $V$ is an ultrafilter over $J$,
then the ultrafilter $U \otimes V$ over $I \times J$ defined by 
$$   X \in    U \otimes V   \eqi  \{ i \in I \,\mid\, \{ j  \in J \,\mid\, \la i, j\ra \in X \} \in V\} \in U    $$
is countably incomplete and $\kappa^+$-good.

The definition of $\otimes$ implies that for every standard $X \in U \otimes U_z$
there is a standard $i \in I$ such that $\{ j  \in J \,\mid\, \la i, j\ra \in X \} \in U_z$, and hence $\la i, z \ra \in X$.
From Bounded Idealization one obtains  $w \in I$ such that $\la w,  z \ra \in X$ holds for all standard  $X \in U \otimes U_z$; in other words, $U_{\la w,z\ra} =  U \otimes U_z $. 
Then $z \in \bbS_{\la w,z\ra}$ and  $\la w,z\ra$ is $\kappa^+$-good, so
$\bbS_{\la w,z\ra}$ satisfies  Bounded Idealization over sets of cardinality $\le \kappa$.
\qed

\begin{remark}
\textbf{More general universes.}\label{general}
The definition of $w$-standard sets in Subsection~\ref{wstandard} can be generalized.
Let $w: S \to I$ where $S, I$ are standard sets.
We let $$\bbS_w =
 \pmbaa f(w(s_1),\ldots,w(s_{k})) \mid  k, f \in \bbS,\,
\dom f = I^k \text{ and }
  s_1,\ldots,s_{k} \in S \,\cap \, \bbS\pmbbb.$$ 
It turns out that these universes correspond precisely to the standard limit ultrapowers of the standard universe.
The proof is similar to the model-theoretic proof that every internal universe ${}^{\ast}V(X)$ is a bounded limit ultrapower of the superstructure $V(X)$; see Chang and Keisler~\cite{CK}, Theorems 4.4.19 and 6.4.10. 
 With suitable modifications,  all results described in this paper remain valid for this more general notion of $w$-standard sets.
\end{remark}


\section{Some earlier constructions} \label{appA}

There are several earlier publications where constructions of nonstandard hulls and Loeb measures in the internal framework are discussed. Below we summarize this work and provide some critical assessment.

\subsection{The ``full'' nonstandard hull} \label{full}
Let $(M, d)$  be a standard metric space.  
A straightforward  attempt to carry out the  construction of the nonstandard hull of $(M, d)$ in $\BST$ can start as follows.
Let 
$$\bbB_{\max}  = \pmbaa x \in M \,\mid\,d(x,a) \text{ is limited for some standard } a \in M\pmbbb.$$
Let $\bbE_{\max}$ be the equivalence relation on $\bbB_{\max}$ defined by 
$$\bbE_{\max} = \pmbaa \la x, y \ra \in \bbB_{\max} \times \bbB_{\max}\,\mid\,  d(x,y) \simeq 0 \pmbbb.$$
Finally, let the  function $\bbD_{\max}$ with standard real values be defined  by 
$$\bbD_{\max} = \pmbaa \la x, y, r \ra \ \in \bbB_{\max} \times \bbB_{\max} \times \bR\,\mid\, r = \sh(d(x,y))\pmbbb.$$
The classes $\bbB_{\max}, \bbE_{\max}$ and $\bbD_{\max}$ are (definable) external sets. 
For every $x \in \bbB_{\max}$, the equivalence class $\bbM(x) = \pmbaa z \in \bbB_{\max} \,\mid\,  d(x,z) \simeq 0 \pmbbb$ is an external set.
However, the final step in the construction of the nonstandard hull of $(M,d)$, to wit, the formation of the quotient space $(\bbB_{\max}/\bbE_{\max}, \bbD_{\max}/\bbE_{\max})$, cannot be carried out in $\BST$ (it would require a ``class of classes'' $\pmbaa \bbM(x) \,\mid\, x \in \bbB_{\max} \pmbbb$).

There are some ways around this difficulty.
Perhaps the most straightforward is to forgo the formation of the quotient space and work with the \emph{representatives} of the  equivalence classes (ie, the elements of $\bbB_{\max}$), and with the congruence $\bbE_{\max}$ in place of the actual equality. This would be similar to working with fractions rather than the rational numbers. But this way does not produce the nonstandard hull as an actual object of $\BST$.  

 The quotient space $(\bbB_{\max}/\bbE_{\max}, \bbD_{\max}/\bbE_{\max})$ can be formed in $\HST$ (using its axiom of Replacement for $\st$-$\in$-formulas). 
An interpretation of $\HST$ can be coded in $\BST$ (see~\cite{KR}, Definition 5.1.2, Theorem 5.1.4 and Corollary 5.1.5), so in this indirect way the ``full'' nonstandard hull can be coded in $\BST$. Unfortunately, the coding involved is far from being a ``morphism'' in any sense, so the resulting opaque code is unsuitable for transferring nonstandard intuitions about the hull to its coded version. One point of working with subuniverses is  that they have a \emph{natural} coding (by standard sets).

Perhaps the most serious objection to this way of constructing nonstandard hulls is that $(\bbB_{\max}/\bbE_{\max}, \bbD_{\max}/\bbE_{\max})$ is just ``too large.'' Trivially,  every nonstandard hull 
$(\bbB_{w}/\bbE_{w}, \bbD_{w}/\bbE_{w})$
of $(M, d)$ considered in Subsection~\ref{nshulls} embeds isometrically into 
$(\bbB_{\max}/\bbE_{\max}, \bbD_{\max}/\bbE_{\max})$ 
(note that $\bbB_w = \bbB_{\max} \cap \bbS_w$, \,$\bbE_w = \bbE_{\max}\cap \bbS_w$, and $\bbD_w =  \bbD_{\max} \cap \bbS_w$).
In many interesting cases, the metric space $(M,d)$ has nonstandard hulls of arbitrarily large cardinality, so $\bbB_{\max}/\bbE_{\max}$ is not of standard size. 
It would  be difficult to do further work with nonstandard hulls using this method, such as compare them with other standard metric spaces, take their products, or form  the space of continuous functions on them.
They are analogous to the ``universal group'' that can be constructed as a direct sum (or product) of all groups. This ``object'' is a proper class  in $\ZFC$, hardly if ever used for more than bookkeeping purposes.
Another important point about working with subuniverses is that the objects produced are standard sets (or external sets of standard size).

\subsection{Vakil's construction}\label{vakil}
In~\cite{Vakil}, Vakil presents a construction of nonstandard hulls of uniform spaces in $\IST$.
 His method  does not require fixing a particular subuniverse, and Standardization is used in a way similar to this paper.  But Vakil's method applies
only to a certain class of uniform spaces, the so-called Henson-Moore spaces. 
Revealingly, these are precisely the spaces whose nonstandard hull is independent of the choice of the nonstandard universe, ie, it is unique up to isomorphism and of standard size; see Henson-Moore~\cite{HM} and Vakil~\cite{V2}.

\subsection{Loeb measures in $\IST$} \label{notloeb}
Diener and Stroyan~\cite{DS}, p. 274, outline a possible construction of Loeb measures in $\IST$, referencing Stroyan and Bayod~\cite{SB}, Section 2.2, for further details. Here we briefly consider this approach.

Let $(\Omega, \cA, \mu)$ be an internal finitely-additive measure space, with $\Omega \subseteq O$ for a standard set $O$. 
We take $\cA = \cP(\Omega)$ for simplicity, and  analyze Loeb's construction from the point of view of $\BST$. The first step is to extend the algebra 
$\cA$ to an external $\sigma$-algebra.
Bounded Idealization in $\BST$ implies that an externally countable union of (internal) sets is either equal to a finite union or is not internal. So the construction has to deal with  external sets from the very beginning. The only way to treat external sets as objects in $\BST$  is via some kind of coding by  sets. For example, every external sequence 
$\pmbaa X_n \,\mid\, n \in \bN \cap \bbS\pmbbb$ of (internal) sets has an extension $\la X_n \,\mid\, n \in \bN \ra$
to an (internal) sequence,\footnote{
This follows from the Extension principle of $\BST$; see Kanovei and Reeken~\cite{KR}, Section 3.2e.
}
 which can be regarded as its code (of course an external sequence has many codes). This coding could be extended to higher levels of the Borel hierarchy over the algebra of (internal) subsets of $\Omega$. A simpler solution, proposed in ~\cite{DS, SB}, is to code Loeb measurable sets with the help of  Souslin schemata. One can define a \emph{Souslin schema} in $\BST$ as a function 
$S : \bN^{<\omega}  \to \cP (\Omega).$
Let $\cF =  \bN^{\omega}$.
The \emph{kernel} of $S$ is the external set 
$$\boldsymbol{\ker} S = \bigcup_{f \in \cF \cap \bbS}\; \bigcap_{n \in \bN \cap \bbS} S_{f \upharpoonright n}.$$
The external sets obtainable as kernels of Souslin schemata are \emph{Henson sets} and Henson sets whose complement in $\Omega$ is also Henson are the \emph{Loeb sets}. Loeb sets form the smallest external $\sigma$-algebra  $\boldsymbol{\sigma}(\cA)$ containing $\cP(\Omega)$; see~\cite{SB}, Theorem 2.2.3 (Luzin Separation Theorem).
Let $\bbL$ be the external set of all pairs $\la S_1, S_2\ra$ such that $\boldsymbol{\ker} S_1 \cap \boldsymbol{\ker} S_2 = \emptyset$ and $\boldsymbol{\ker} S_1 \cup \boldsymbol{\ker} S_2 = \Omega$. Then $\bbL$ can be viewed as a set of codes for Loeb sets. The algebraic operations (union, complement, externally countable union) can be coded by external relations on $\bbL$, and the Loeb measure itself can be coded  by an external  relation
on $\bbL \times (\bR \cap \bbS)$.  

The objections raised  in Subsection~\ref{full} against ``full'' nonstandard hulls apply also to the above approach to Loeb measures. In particular, one cannot form the quotient space of $\bbL$ modulo the  relation in which two codes are equivalent when they code the same external set, and the coding of the set-theoretic operations is not a ``morphism'' (eg, the code of the union of two sets is in no sense the union of their codes).
It would be awkward to work with random variables on $\bbL$ via codes, and collections of random variables are even more challenging.
 In addition, it is customary in the literature (see Albeverio et al.~\cite{A}, Kanovei and Reeken~\cite{KR}) to regard not $\boldsymbol{\sigma}(\cA)$ but its \emph{completion} $\boldsymbol{L}(\cA)$ as \emph{the} Loeb $\sigma$-algebra.  It is not usually possible to extend $\bbL$ to an external set of codes for 
$\boldsymbol{L}(\cA)$ (not even in $\HST$, because the Power Set axiom for external sets does not hold there). 
On these grounds, it is arguable whether the claim that this method represents the Loeb measure space is justified.

\subsection{Bounded Internal Set Theory}\label{Bist}
Following upon ideas of Lindstr\o m~\cite{L}, Diener and Stroyan~\cite{DS} aim for ``a description of the common ground between the two approaches to Robinson's theory.'' This takes the form of working in a polysaturated\footnote{
$\fN$ is \emph{polysaturated} if it is $\kappa$-saturated for $\kappa = |V(\as S)|$.
} nonstandard universe $\fN = ( V(S), V(\as S), \ast)$ axiomatically.
They formulate Bounded Internal Set Theory (b$\IST$) and show that its axioms hold in such universes. 
This theory should not be confused with $\BST$.  The main differences are:
\begin{itemize}
\item
The language of b$\IST$ is closely tied to the structure $\fN$. It contains a constant symbol for every internal set in $V(\as S)$ and more; in particular, it is  uncountable.
\item
The axiom schemata T, I and S are modified so as to apply only to those formulas in which all quantifiers are bounded.
\item
The axioms of $\ZFC$ are not postulated (in fact, Replacement fails in $V(S)$).
\end{itemize}

b$\IST$ proves many results familiar from $\IST$ or $\BST$, but, just like in these theories, its variables range over internal sets. It does not provide access to higher order external sets without some coding. 
Of course, the nonstandard universe $\fN$ does provide external sets that can be used to construct nonstandard hulls and Loeb measures in the usual way, but then one is using the model-theoretic rather than the axiomatic approach. 
Overall, b$\IST$ is more a useful tool for work with superstructures than a self-standing axiomatics for nonstandard analysis.

\subsection{Measure and integration over finite sets}
Nelson's  Radically Elementary Probability Theory~\cite{N2}  works with (hyper)finite probability spaces only.  
It requires only very elementary axioms (see~\cite{N2}, Chapter 4) that are easy consequences of the axioms of $\IST$,
 or even of the weaker and more
effective theory SCOT (see footnote \ref{f6}).
Further development of this approach can be found eg. in Cartier and Perrin~\cite{CP1, CP2}, who  study both 
Nelson's S-integral on finite measure spaces (which they call Loeb-Nelson integral) and its refinement, which they call Lebesgue integral. The ``radically elementary'' approach is simple and  elegant, and many interesting results have been obtained in this way. 

It is
beyond the scope of this paper to try to compare the ``radically elementary'' approach to the  nonstandard measure theory based on the work of Loeb.  
 Nelson shows (\cite{N2}, A.1) that for every stochastic process there is a \emph{nearby elementary process}, and that theorems of the conventional theory of stochastic processes can be derived from their elementary analogues. The works~\cite{N2, CP1, CP2}  do not address the question whether (up to isomorphism) Loeb measure spaces can be obtained from their constructions.


\section{Conclusion}

This paper is based on the fact that an internal set theory such as $\BST$ provides a supply of structures that are analogous to the internal part of model-theoretic nonstandard universes. 
Utilizing a natural coding and the principle of Standardization, we can mimic  external sets by standard ones.
Nonstandard hulls and Loeb measures can be obtained ``up to an isomorphism''  with their external counterparts. 
In $\BST$ they are standard objects in a common framework, which makes it easy in principle to compare them.  
In the end, they are the same structures that could be obtained by the superstructure methods carried out in $\BST$, but our technique for obtaining them is more suitable for the internal axiomatic framework.

The theory $\BST$ is the internal part of $\HST$, a theory which axiomatizes  external sets, in addition to standard and internal ones. In this theory one can carry out external constructions in the same way as in nonstandard universes; no coding is needed. The coding can be introduced afterwards in order to convert the external structures so obtained to  standard structures.

We conclude that the internal axiomatic approach based on $\BST$ can implement virtually all techniques employed by the model-theoretic approach, with some advantages: 
It could make nonstandard hulls and Loeb measures more easily accessible to working mathematicians who have learned the nonstandard methods in $\IST$ or another axiomatic framework,
 because it does not require  a priori knowledge of ultrafilters and ultrapowers or model theory. 
It establishes a single unified framework in which both the standard ``world'' and the nonstandard ``world'' are adequately axiomatized. It allows a seamless extension to a theory that encompasses also external sets. 
Finally, it enables reverse-mathematical analysis of the strength of the Axiom of Choice 
(see eg. Hrbacek and Katz~\cite{HK})
and other axioms used in the practice of nonstandard analysis.


\bibliographystyle{amsalpha}
\end{document}